\documentclass[11pt]{article}
\usepackage{graphicx}
\usepackage{amsmath,amssymb,amsthm}
\usepackage{mathtools}
\usepackage{xcolor}
\usepackage{caption}
\usepackage{subcaption}
\usepackage{comment}
\usepackage{hyperref}
\hypersetup{colorlinks,breaklinks,
	linkcolor=blue,urlcolor=blue,
	anchorcolor=blue,citecolor=blue}
\usepackage[nameinlink, capitalise, noabbrev]{cleveref}
\usepackage{dsfont}
\usepackage{nicematrix}
\usepackage{enumitem}
\usepackage{kotex}
\usepackage{mathrsfs}
\usepackage{overpic}
\usepackage{tikz}
\usepackage{makecell}
\usepackage[table]{xcolor}

\DeclareMathOperator*{\argmin}{arg\,min}

\newtheorem{theorem}{Theorem}[section]

\newtheorem{lemma}[theorem]{Lemma}  
\newtheorem{proposition}[theorem]{Proposition}
\newtheorem{corollary}[theorem]{Corollary}
\newtheorem{remark}[theorem]{Remark}
\numberwithin{equation}{section}
\newtheorem{result}{Result}[section]

\title{Spectral Selection in Symmetric Self-Attention Dynamics}
\author{
Christian Kuehn\\
Department of Mathematics, Technical University of Munich\\
Munich Data Science Institute (MDSI)\\
Munich Center for Machine Learning (MCML)\\
\texttt{ckuehn@ma.tum.de}
\and
Jaeyoung Yoon\thanks{Corresponding author.}\\
Department of Mathematics, Technical University of Munich\\
\texttt{wodud1516@gmail.com}
}

\begin{document}
	
	\maketitle
    % Physica D, SIAM DS, JNS, Nonlinear Dynamics
    \begin{abstract}
        We study self-attention dynamics on the unit sphere as an interacting particle system arising from an idealized Transformer-type update. Under a symmetry assumption on weight matrices given by $Q^\top K=V=V^\top$, the flow admits a gradient-flow structure and an exact reformulation in the eigenbasis of $V$, revealing a spectral mode-selection mechanism. We show that the dynamics exhibits two distinct asymptotic scenarios: homogeneous alignment toward the dominant eigendirection when one positive eigenvalue strictly dominates all others in modulus, and sign-split polarization toward the most negative eigendirection when $V$ is negative definite. In particular, we obtain local stability criteria for pure-mode equilibria and global selection results in both regimes. These results provide a rigorous finite-particle description of how the spectrum of the weight matrices organizes asymptotic patterns in a symmetric self-attention flow, and highlight how the symmetric setting renders the dynamics amenable to mathematical analysis.
    \end{abstract}
%============================================================================================================================
%============================================================================================================================
\section{Introduction}

Self-attention is one of the central mechanisms in modern Transformer architectures, and its repeated application induces a nonlinear collective evolution of token representations across depth. Since the introduction of the Transformer architecture in \cite{Vaswani2017}, this mechanism has primarily been studied from the viewpoint of machine learning and representation design~\cite{Lin2022}. This mathematical perspective on neural architectures has precedents: dynamical-systems viewpoints have long been used to study neural-network models \cite{Hopfield1982,Cohen1983,Sontag1998,E2017}, and residual networks have been interpreted as neural ODEs \cite{Chen2018,Haber2018,Ruthotto2020}, leading to further geometric and expressivity analyses of neural ODE architectures \cite{KK2023,KK2026}. More recently, this viewpoint has been extended to Transformer architectures \cite{GLPR2023,GLPR2025,CACP2025}: rather than treating attention only as an algorithmic component, one may regard it as generating a collective dynamics of interacting tokens and ask qualitative questions about the resulting evolution, such as clustering, alignment, collapse, polarization, and long-time pattern formation. This viewpoint is particularly natural when one is interested in the forward representation dynamics induced by repeated attention layers~\cite{GLPR2023,GLPR2025}.

A particularly useful setting arises when normalization is idealized by constraining token representations to evolve on the unit sphere. This spherical viewpoint is also motivated by a root-mean-square-type normalization, which regulates the scale of token representations and, in continuous-depth idealizations, naturally leads, after a suitable rescaling, to dynamics on a normalized state space \cite{GLPR2025,KGPR2025}. In this setting, the depth evolution of a stack of attention layers can be modeled as a self-attention interacting particle system on $\mathbb S^{d-1}$. The spherical formulation is attractive both mathematically and conceptually. Mathematically, it provides a geometrically structured state space and, in special regimes, admits variational or gradient-flow formulations. Conceptually, it captures representation dynamics that remain closely tied to clustering and collapse phenomena observed in deep attention models \cite{GLPR2025,KGPR2025}.

In this paper, we study a finite-particle self-attention dynamics on the unit sphere under a structural symmetry assumption. More precisely, starting from an idealized Transformer-type update in which the feed-forward layer is omitted and the normalization step is replaced by spherical normalization, one is formally led to the system
\begin{equation*}
\dot{x}_i
=
P_{x_i}^{\perp}
\left(
\frac{1}{Z_{\beta,i}}
\sum_{j=1}^n e^{\beta\langle Qx_i,Kx_j\rangle} Vx_j
\right),
\quad
Z_{\beta,i}:=\sum_{k=1}^n e^{\beta\langle Qx_i,Kx_k\rangle},
\quad i\in[n]:=\{1,\cdots,n\},
\end{equation*}
where $x_i=x_i(t)\in \mathbb S^{d-1}$ denotes the representation of the $i$-th token, $Q,K,V\in\mathbb R^{d\times d}$ are the query, key, and value matrices, respectively, $\beta>0$ is the inverse-temperature parameter, $Z_{\beta,i}$ is the softmax normalizing factor, and $P_{x_i}^{\perp}$ denotes the orthogonal projection onto the tangent space $T_{x_i}\mathbb S^{d-1}$. Thus, each token evolves under a projected attention-weighted average of the value vectors, and the resulting dynamics may be viewed as a self-attention interacting particle system on $\mathbb S^{d-1}$.

Our analysis is carried out under the symmetric assumption
\begin{align}\label{symm_assu}
Q^\top K = V = V^\top.
\end{align}
In this regime, the attention score is determined by the same symmetric matrix $V$ that also drives the value interaction, and the system acquires additional structure. In particular, the dynamics admits a gradient-flow formulation and becomes amenable to a spectral analysis in the eigenbasis of $V$. This makes the symmetric setting a natural testing ground for understanding how the interaction spectrum organizes the long-time behavior of the self-attention flow.

Recent mathematical studies of self-attention and Transformer-type dynamics have developed along several related directions. A first line of work introduced continuous-time and interacting-particle viewpoints for self-attention and showed that even simplified finite-particle models already exhibit nontrivial clustering behavior \cite{GLPR2023}. This perspective was subsequently broadened and systematized in \cite{GLPR2025}, where Transformer architectures are interpreted through interacting particle systems, continuum limits, and related dynamical frameworks. See also \cite{ChemnitzEngelKuehnKuntz} for a broader dynamical-systems perspective on neural-network architectures. In parallel, more general PDE and continuum descriptions have also been developed to place deep Transformer dynamics into a wider analytical setting \cite{CACP2025}.

A second line of work has emphasized the role of normalization and variational structure. In normalized or spherical regimes, mean-field formulations and Wasserstein-type perspectives make it possible to analyze attention dynamics through PDE and gradient-flow methods \cite{BKKRW2025,GLPR2025,Rigollet2025}. More recent studies have further shown that the long-time behavior is richer than a simple one-step collapse picture, revealing metastable clustering, multiscale evolution, and normalization-dependent effects \cite{GKPR2024,BPA2025_a,BPA2025_b,KGPR2025}. Taken together, these works show that self-attention dynamics already supports a substantial mathematical theory, and that structured settings such as normalized or symmetric regimes provide a particularly useful baseline for rigorous dynamical analysis.

Against this background, the present paper focuses on a different but closely related question: finite-particle spectral mode selection under the symmetric assumption \eqref{symm_assu}. More specifically, we ask:
\begin{center}
    Under the symmetric assumption \eqref{symm_assu}, which eigendirections of $V$ are selected by the dynamics, and in what geometric form does this selection appear?
\end{center}
A central message of the paper is that the sign structure of the spectrum leads to genuinely different selection mechanisms. More precisely, we distinguish two regimes in terms of the eigenvalues 
\(\{\lambda_i\}_{i=1}^d\) of \(V\): the \textit{positive-dominant regime} where $\lambda_1>\max_{k\ge2}|\lambda_k|$, and the \textit{negative-definite regime} where $\lambda_k<0$ for all $k\in[d]$. In the positive-dominant regime, the global selection mechanism is alignment-driven and relies on a one-sided cone structure. In contrast, in the negative-definite two-particle regime, the geometry is anti-alignment-driven and the dynamics is led toward sign-split configurations. Thus, even within the symmetric setting, the long-time behavior is governed not only by the dominance of an eigenvalue, but also by the way attention interacts with the geometry of the sphere and with the sign pattern of the spectrum.

The contribution of the paper is to make this spectral-selection mechanism explicit at the level of the full finite-particle system. First, we reformulate the symmetric self-attention dynamics in the eigenbasis of $V$ and derive a nonlocal replicator-type system for the modal variables, which makes the competition among eigendirections transparent. In fact, to determine a concrete link to replicator equations is itself already an interesting observation. Second, we identify two invariant manifolds on which the dynamics closes, namely the consensus manifold and a balanced bipolar manifold, and analyze the corresponding reduced systems. Third, motivated by these reduced dynamics, we study pure-mode equilibria of the full system and characterize their local stability, including the dependence of sign-split stability on the attention sharpness and the imbalance between the two sign groups. Finally, we establish global mode-selection results in two representative regimes: a positive-dominant regime, where a forward-invariant cone yields convergence to the leading positive eigendirection, and a two-particle negative-definite regime, where the dynamics becomes asymptotically sign-split and generically selects the eigendirection corresponding to the smallest eigenvalue.

The viewpoint of the paper is entirely finite-particle. Rather than passing first to a mean-field description, we exploit the exact geometric and spectral structure available in the symmetric system. In this sense, the paper identifies a class of self-attention dynamics for which the long-time behavior can be analyzed by combining a variational interpretation, a modal reformulation, invariant-manifold reductions, and stability arguments. More broadly, we hope this contributes to the emerging mathematical study of Transformer-inspired dynamics by showing that, in suitably structured regimes, questions originating from attention mechanisms can lead to precise problems in dynamical systems, collective behavior, and spectral and nonlinear analysis.

The rest of the paper is organized as follows. In \Cref{sec_prelim}, we collect the geometric and variational preliminaries for the symmetric self-attention dynamics. In \Cref{sec_modal}, we derive a modal reformulation in the eigenbasis of $V$ and show that the resulting dynamics has a replicator-type structure. In \Cref{sec_reduced}, we study reduced dynamics on invariant manifolds and identify the corresponding mode-selection mechanisms. In \Cref{sec_stab}, we analyze pure-mode equilibria and their local stability. In \Cref{sec_global}, we prove global mode-selection results in both positive-dominant and negative-definite regimes. \Cref{sec_numer} presents numerical examples that illustrate the analytical picture developed above and explore nearby regimes beyond the scope of the present theory. Finally, \Cref{sec_conclusion} summarizes the main conclusions and discusses several directions for future work. Several long calculations required for the  proofs are deferred to the appendices.

%============================================================================================================================
%============================================================================================================================
    \section{Preliminaries}\label{sec_prelim}
    In this section, we collect the background material used throughout the paper. In \Cref{subsec:ips_derivation}, we introduce the self-attention interacting particle system on the sphere that serves as the main finite-particle model studied in this work. In \Cref{subsec:present_results}, we describe the gradient-flow structure induced by the symmetric assumption and briefly summarize the dynamical results in the literature that are most relevant to the questions addressed in this paper.
%============================================================================================================================
    \subsection{From self-attention to an interacting particle system on the sphere}\label{subsec:ips_derivation}

    We briefly recall the continuous-time modeling viewpoint for self-attention dynamics on the sphere. 
    Our starting point is an idealized Transformer block in which the feed-forward layer is omitted and the normalization step is replaced by a normalization onto the unit sphere. 
    This simplified setting retains the self-attention mechanism while leading to a tractable finite-particle dynamical system.
    
    Consider an $n$ tokens representation
    \[
    x_1^{(\ell)},\cdots,x_n^{(\ell)} \in \mathbb{R}^d,
    \]
    at layer \(\ell=0,\cdots,L\). For each token \(x_i^{(\ell)}\), the self-attention weights assigned to the tokens \(x_j^{(\ell)}\) are defined by
    \[
    K_{ij}^{(\ell)}
    :=
    \frac{\exp\!\big(\beta \langle Qx_i^{(\ell)},Kx_j^{(\ell)}\rangle\big)}
    {\sum_{m=1}^n \exp\!\big(\beta \langle Qx_i^{(\ell)},Kx_m^{(\ell)}\rangle\big)},
    \quad i,j\in[n],
    \]
    where \(Q\), \(K\), and \(V\) denote the query, key, and value matrices, respectively, and \(\beta>0\) is an attention sharpness parameter. The corresponding attention output is
    \[
    \mathcal A_i\big(\mathbf{x}^{(\ell)}\big)
    :=
    \sum_{j=1}^n K_{ij}^{(\ell)}\,Vx_j^{(\ell)},
    \quad
    \mathbf{x}^{(\ell)}:=(x_1^{(\ell)},\cdots,x_n^{(\ell)}).
    \]
    Thus, each token interacts nonlinearly with all the others, with interaction weights determined by the query--key compatibilities.
    
    We next consider the sphere-normalized (or $L^2$-normalized) residual update
    \[
    x_i^{(\ell+1)}
    =
    \frac{x_i^{(\ell)}+h\,\mathcal A_i(\mathbf{x}^{(\ell)})}
    {\|x_i^{(\ell)}+h\,\mathcal A_i(\mathbf{x}^{(\ell)})\|},
    \quad\forall~i\in[n],
    \]
    where \(h>0\) is a step size.
    This normalization may be viewed as an idealized counterpart of root-mean-square (RMS) normalization \cite{Zhang2019,Touvron2023}:
    \[\text{RMSNorm}(x)=\frac{x\odot g}{\sqrt{\frac1d\|x\|^2+\varepsilon}},\]
    where \(g\in\mathbb R^d\) is a learnable gain parameter and \(\varepsilon>0\) is a numerical stabilizer. In the idealized case \(g\equiv\mathbf 1\) and \(\varepsilon=0\), this reduces to
    \[\text{RMSNorm}(x)=\sqrt d\,\frac{x}{\|x\|}.\]
    Ignoring the learnable scaling $g$ and the numeric stabilizer $\varepsilon$, RMS normalization differs from \(L^2\)-normalization only by a constant scaling factor $\sqrt{d}$. 
    This makes spherical normalization a natural and geometrically convenient simplification.
    
    Identifying the depth variable with the continuous time \(t=\ell h\) and letting \(h\to0\), one is formally led to the system
    \begin{align}\label{TF_ODE_recalled}
    \dot x_i
    =
    P_{x_i}^\perp\left(
    \frac{1}{Z_{\beta,i}}
    \sum_{j=1}^n e^{\beta\langle Qx_i,Kx_j\rangle}Vx_j
    \right),
    \quad
    Z_{\beta,i}
    =
    \sum_{k=1}^n e^{\beta\langle Qx_i,Kx_k\rangle},
    \end{align}
    for \(i\in[n]\), where
    \[
    P_x^\perp y := y-\langle x,y\rangle x
    \]
    denotes the orthogonal projection onto the tangent space \(T_x\mathbb S^{d-1}\) at point $x$ in the $(d-1)$-dimensional unit sphere $\mathbb S^{d-1}$.
    Accordingly, \eqref{TF_ODE_recalled} may be viewed as an interacting particle system induced by self-attention on the unit sphere. In practical implementations, the inner product in the attention score is often accompanied by the factor \(d^{-1/2}\). Throughout the paper, this factor is absorbed into the parameter \(\beta\). Regarding the structure of~\eqref{TF_ODE_recalled}, we point out that it can be viewed as a variant of Kuramoto-type interacting particle systems~\cite{Kuramoto,Strogatz1,PikovskyRosenblumKurths,Acebronetal}, just with a more complex interaction and normalization mechanism. In fact, it was pointed out in detail in~\cite{ChemnitzEngelKuehnKuntz} that extremely large classes of neural network architectures can be interpreted as Kuramoto-type network dynamical systems.
    
    In this paper, we study \eqref{TF_ODE_recalled} and further impose the structural assumption
    \[
    Q^\top K = V,
    \quad
    V^\top = V.
    \]
    Under this symmetry,
    \[
    \langle Qx_i,Kx_j\rangle
    =
    \langle x_i,Q^\top Kx_j\rangle
    =
    \langle x_i,Vx_j\rangle,
    \]
    and hence \eqref{TF_ODE_recalled} reduces to
    \begin{align}\label{TF_ODE_sym}
    \dot x_i
    =
    P_{x_i}^\perp\left(
    \frac{1}{Z_{\beta,i}}
    \sum_{j=1}^n e^{\beta\langle x_i,Vx_j\rangle}Vx_j
    \right),
    \quad
    Z_{\beta,i}
    =
    \sum_{k=1}^n e^{\beta\langle x_i,Vx_k\rangle}.
    \end{align}
    Therefore, in this symmetric setting, the dynamics is completely determined by the real symmetric interaction matrix \(V\).
%============================================================================================================================
    \subsection{Gradient-flow structure and related dynamical results}\label{subsec:present_results}
    We recall the weighted gradient formulation of \cite{GLPR2025}. Under the symmetric assumption
    \[
    Q^\top K = V = V^\top,
    \]
    the self-attention dynamics \eqref{TF_ODE_sym} can be written as a gradient system after introducing a state-dependent weighted metric. More precisely, for a configuration
    \[
    X=(x_1,\cdots,x_n)\in (\mathbb S^{d-1})^n,
    \]
    we define the interaction energy
    \[
    E_\beta(X)
    := \frac{1}{2\beta}\sum_{i=1}^n\sum_{j=1}^n e^{\beta\langle x_i,Vx_j\rangle},
    \]
    and equip the tangent space
    \[
    T_X(\mathbb S^{d-1})^n
    =
    T_{x_1}\mathbb S^{d-1}\times\cdots\times T_{x_n}\mathbb S^{d-1}
    \]
    with the weighted inner product
    \[
    \langle A,B\rangle_X^{(\beta)}
    :=
    \sum_{i=1}^n Z_{\beta,i}(X)\,\langle a_i,b_i\rangle,
    \quad
    A=(a_i)_{i=1}^n,\quad B=(b_i)_{i=1}^n,
    \]
    where
    \[
    Z_{\beta,i}(X):=\sum_{j=1}^n e^{\beta\langle x_i,Vx_j\rangle}.
    \]
    With this metric, the computation gives
    \[\operatorname{grad}_{\beta}E_\beta(X)_i=\frac{1}{Z_{\beta,i}(X)}P_{x_i^\perp}\left(\sum_{j=1}^n e^{\beta\langle x_i,Vx_j\rangle}Vx_j\right),\]
    where we denote by \(\operatorname{grad}_{\beta}\) the Riemannian gradient with respect to the weighted metric \(\langle\cdot,\cdot\rangle_X^{(\beta)}\). This coincides with the vector field in \eqref{TF_ODE_sym}, hence, the system \eqref{TF_ODE_sym} has a weighted gradient structure with respect to $\langle\cdot,\cdot\rangle_X^{(\beta)}$.
    
    For comparison, if one removes the normalization factor \(Z_{\beta,i}^{-1}\), one obtains a closely related unnormalized self-attention model. In the isotropic setting ($Q^\top K=V=I$), its continuum formulation is discussed in \cite{Rigollet2025} as a Wasserstein gradient flow of the same interaction energy. Thus, the normalized system \eqref{TF_ODE_sym} may be viewed as a state-dependent metric counterpart of a closely related unnormalized gradient-flow model.

We next recall several representative dynamical results for symmetric self-attention models, mostly in the isotropic setting, that are closely related to \eqref{TF_ODE_sym}. These results provide useful context for the present work, although their main focus is clustering, synchronization, and mean-field concentration, rather than finite-particle spectral selection. Technical assumptions are omitted for brevity and we refer the reader to the cited references for precise statements.

\begin{result}[Asymptotic clustering in self-attention dynamics {\cite[Thm. 4.2, 4.3, 5.1, 6.1, 6.3]{GLPR2025}}]
    Consider the dynamics \eqref{TF_ODE_sym} with arbitrary query and key matrices $Q,K$ and $V=I_d$. If the initial configuration lies in an open hemisphere, all tokens converge exponentially to a single cluster. For $d\ge n$ and uniformly sampled initializations, this condition holds almost surely. More generally, clustering holds for Lebesgue-almost every initial configuration across a range of regimes, including arbitrary $\beta>0$ when $d\ge3$, and extreme values of $\beta$ for arbitrary $d$ and $n$.
\end{result}

While clustering is guaranteed for almost every initial configuration when $V=I_d$, in sharper attention regimes ($\beta\gg1$), it may be preceded by long-lived metastable configurations.

\begin{result}[Dynamic metastability for separated configurations {\cite[Thm.~1.2]{GKPR2024}}]
For sufficiently sharp attention and initially separated configurations, the self-attention flow exhibits a metastable clustering regime. More precisely, particles initially belonging to the same spherical cap become exponentially close after a transient time and remain trapped near their initial caps for an exponentially long time interval.
\end{result}

Continuum and mean-field formulations provide another perspective on these phenomena. Most rigorous results in this direction concern mean-field attention dynamics in the isotropic setting \(Q=K=V=I_d\), where variational methods can be used to study concentration toward Dirac measures.

\begin{result}[Quantitative clustering for isotropic mean-field attention {\cite[Thm.~2.4]{CLPR2025}}]
Consider the isotropic mean-field attention dynamics on \(\mathbb S^{d-1}\). If the initial measure has nonzero mean and admits an \(L^2\)-density with respect to the uniform measure, then for sufficiently small \(\beta>0\) the solution converges exponentially, in \(W_2\)-distance after a transient time, toward a Dirac mass on the sphere.
\end{result}

Mean-field models also capture multi-cluster metastable regimes, linking continuum descriptions with the finite-particle metastability observed in attention dynamics.

\begin{result}[Clustering phase in the metastable unnormalized mean-field attention model {\cite[Thm.~4.5]{BPA2025_a}}]
In the metastable mean-field regime for the unnormalized ($Z_{\beta,i}\equiv 1$) attention model, the empirical particle distribution at the clustering time is approximated by the corresponding mean-field solution. Under the assumptions of the cited work, the limiting measure is close, in \(W_1\)-distance and in probability, to a finite sum of Dirac masses.
\end{result}

The results recalled above emphasize clustering, synchronization, metastability, and mean-field concentration mechanisms; for the detailed technical statements we refer to the respective papers. In contrast, the present paper remains at the finite-particle level and exploits the spectral structure of the symmetric matrix \(V\) to study how individual eigendirections are selected by the dynamics.
%============================================================================================================================
%============================================================================================================================
	\section{Symmetric self-attention dynamics and modal formulation}\label{sec_modal}
    In this section, we reformulate the symmetric self-attention dynamics \eqref{TF_ODE_recalled} in coordinates adapted to the spectral structure of the interaction matrix $V=Q^\top K$. This reduction is the natural starting point of the analysis, since in the symmetric setting ($V^\top=V$) the nonlinear attention dynamics \eqref{TF_ODE_sym} can be expressed in terms of modal interactions between eigendirections. The resulting formulation reveals a nonlocal competition mechanism among modes and provides the basic variables for the reduced dynamics and stability analysis developed later.

    Since \(V\) is real symmetric, one may diagonalize it by an orthogonal change of variables. Thus, if desired, the dynamics can be rewritten in a diagonal basis. However, in what follows we keep the formulation for a general symmetric matrix and only use the spectral decomposition of \(V\) to derive the modal equations.
	%============================================================================================================================

	\subsection{Coordinate dynamics}
    Let \((e_k)_{k=1}^d\) be an orthonormal eigenbasis of \(V\), with corresponding eigenvalues \((\lambda_k)_{k=1}^d\), so that
    \[
        V=\sum_{k=1}^d \lambda_k e_k e_k^\top.
    \]
    We write each trajectory \(x_i(t)\in\mathbb S^{d-1}\) in the form
    \[
        x_i(t)=\sum_{k=1}^d c_{i,k}(t)e_k,\quad\forall~i\in[n].
    \]
    Our goal is to derive a closed evolution equation for the modal coefficients \((c_{i,k})_{i,k}\), which makes the role of the spectrum of \(V\) explicit. In these coordinates, the attention weights become
    \[
    K_{ij}(C):=\frac{\exp\!\left(\beta\sum_{l=1}^d \lambda_l c_{i,l}c_{j,l}\right)}{\sum_{m=1}^n\exp\!\left(\beta\sum_{l=1}^d \lambda_l c_{i,l}c_{m,l}\right)},
    \]
    where \(C=(c_{i,k})\in\mathbb R^{n\times d}\). Using \eqref{TF_ODE_sym} and
    \[
    P_{x_i}^\perp y = y-\langle x_i,y\rangle x_i,
    \]
    a direct computation yields
    \begin{align}\label{C_ODE}
        \dot c_{i,k}=\lambda_k \sum_{j=1}^n K_{ij}(C)c_{j,k}-\phi_i(C)c_{i,k},\quad\forall~k\in[d],~~i\in[n],
    \end{align}
    where
    \[
        \phi_i(C):=\sum_{l=1}^d c_{i,l}\lambda_l    \left(\sum_{j=1}^n K_{ij}(C)c_{j,l}\right).
    \]
    In matrix form, introducing
	\[
	C=(c_{i,l})\in\mathbb R^{n\times d},
	\quad
	\Lambda=\mathrm{diag}(\lambda_1,\cdots,\lambda_d),
	\quad
	K=(K_{ij})\in\mathbb R^{n\times n},
	\]
	the system \eqref{C_ODE} can be written compactly as
	\[
	\dot C
	=
	KC\Lambda-\mathrm{Diag}(KC\Lambda C^\top)\,C,
	\]
	where $\text{Diag}(A)$ denotes the diagonal matrix consisting of the diagonal entries of $A$. This form makes clear that the symmetric self-attention dynamics \eqref{TF_ODE_sym} is a nonlocal replicator-type dynamics on the modal coefficients. For more detailed references on the theory of replicator equations we refer to~\cite{Hofbauer1998,Sandholm2010}.
	%============================================================================================================================

	\subsection{Token-wise modal masses}
    For each token \(i\in[n]\) and mode \(k\in[d]\), define the token-wise modal mass
    \[
        a_{i,k}:=c_{i,k}^2.
    \]
    Since \(\sum_{k=1}^d a_{i,k}=|x_i|^2=1\), the vector
    \[
        a_i=(a_{i,1},\cdots,a_{i,d})\in\Delta^{d-1}
    \]
    belongs to the standard unit simplex. In these variables, the modal dynamics takes a replicator-type form, which makes the competitive structure among active modes explicit.

    \begin{proposition}[Replicator-type dynamics for token-wise modal masses]
    For each \(i\in[n]\) and \(k\in[d]\) such that \(c_{i,k}\neq 0\), the system \eqref{C_ODE} is reduced to
    \[
    \dot a_{i,k}
    =
    2a_{i,k}\bigl(f_{i,k}-\overline f_i\bigr),
    \]
    where
    \[
    f_{i,k}
    :=
    \lambda_k
    \frac{\sum_{j=1}^n K_{ij}(C)c_{j,k}}{c_{i,k}},
    \quad
    \overline f_i
    :=
    \sum_{l=1}^d a_{i,l}f_{i,l}
    =
    \phi_i(C).
    \]
    Hence, for each fixed \(i\), the vector \(a_i\) evolves according to a replicator-type equation with state-dependent nonlocal fitness.
    \end{proposition}
    \begin{proof}
    Differentiating \(a_{i,k}=c_{i,k}^2\) and using \eqref{C_ODE}, we obtain
    \[
    \dot a_{i,k}=2c_{i,k}\dot c_{i,k}
    =
    2c_{i,k}\left(
    \lambda_k \sum_{j=1}^n K_{ij}(C)c_{j,k}
    -
    \phi_i(C)c_{i,k}
    \right),
    \]
    which is equivalent to the desired equation.
    \end{proof}
    
	Being given the variables $(c_{i,k})_{i,k}$, a natural averaged quantity (or observable, or order parameter) is
	\[
	m_k(t):=\frac1n\sum_{i=1}^n c_{i,k}^2(t)=\frac1n\sum_{i=1}^n a_{i,k}(t).
	\]
	Although \(m_k\) measures the total mass carried by mode \(e_k\), its evolution is not closed.
	
	\begin{remark}[Lack of closure for the averaged masses]
		In general,
		\[
		\dot m_k
		=
		\frac{2}{n}\sum_{i=1}^n a_{i,k}(f_{i,k}-\overline f_i),
		\]
		and the right-hand side depends on the full configuration \(C=(c_{i,l})\) through the nonlinear attention weights \(K_{ij}(C)\). Therefore, unlike in classical finite-dimensional replicator systems, the dynamics of \((m_k)_k\) cannot be reduced to a closed ODE involving only the averaged masses.
	\end{remark}

    Since the averaged masses do not satisfy a closed evolution equation in general, the modal reformulation alone does not yet determine the asymptotic selection mechanism of the full system. This motivates the two complementary steps pursued in the remainder of the paper. We first isolate invariant manifolds on which the dynamics closes and yields explicit reduced equations. We then return to the full system and analyze pure-mode equilibria together with their stability, which will provide the basis for the global mode-selection results established later.
%============================================================================================================================
%============================================================================================================================

    \section{Reduced dynamics on invariant manifolds}\label{sec_reduced}

    Although the full system \eqref{TF_ODE_sym} is nonlocal and high-dimensional, its mode-selection mechanisms \eqref{C_ODE} become much more transparent on certain invariant manifolds. In this section, we focus on two natural configurations: the \textit{consensus manifold}, which describes homogeneous alignment, and a \textit{balanced bipolar manifold}, which describes polarized states with two opposite orientations. On each of these manifolds, the modal dynamics closes and yields a tractable reduced system. These reduced models provide the first indication of how the spectrum of \(V\) influences asymptotic mode selection.
    %============================================================================================================================

    \subsection{Consensus manifold}\label{subsec:css_mfd}
    
    We begin with the fully aligned regime
    \[
    x_i(t)=x(t),\quad\forall~i\in[n],
    \]
    in which all tokens evolve identically on the sphere. This manifold
    \[\mathcal M_{\mathrm{css}}:=\big\{X\in(\mathbb S^{d-1})^n~:~x_i=x_j,\quad\forall~i,j\in[n]\big\}\]
    is invariant under the flow, and the dynamics reduces to a single trajectory in \(\mathbb S^{d-1}\). When expressed in the eigenbasis of \(V\), the corresponding modal masses satisfy a closed replicator equation, which makes the mode-selection mechanism explicit. Writing
    \[
    x(t)=\sum_{k=1}^d c_k(t)e_k,
    \]
    we obtain the reduced system
    \begin{align}\label{one_part_ODE}
    \dot c_k
    =
    c_k\left(\lambda_k-\sum_{l=1}^d \lambda_l c_l^2\right),
    \quad k\in[d].
    \end{align}
    Introducing
    \[
    p_k:=c_k^2,
    \quad \sum_{k=1}^d p_k=1,
    \]
    we deduce that
    \begin{align}\label{one_part_p_ODE}
    \dot p_k
    =
    2p_k\left(\lambda_k-\sum_{l=1}^d \lambda_l p_l\right),
    \quad\forall~k\in[d],
    \end{align}
    which is a replicator equation.
    
    \begin{proposition}[Explicit solution on the consensus manifold $\mathcal M_{\mathrm{css}}$]
    Let \(x=(c_k)\) solve \eqref{one_part_ODE}, and define \(p_k=c_k^2\). Then
    \begin{align}\label{pk_formula}
        p_k(t)=\frac{p_k(0)e^{2\lambda_k t}}{\sum_{l=1}^d p_l(0)e^{2\lambda_l t}},\quad\forall~k\in[d].
    \end{align}
    Equivalently,
    \begin{align}\label{consensus_sol}
        c_k(t)=\operatorname{sgn}(c_k(0))\left(\frac{c_k(0)^2 e^{2\lambda_k t}}{\sum_{l=1}^d c_l(0)^2 e^{2\lambda_l t}}\right)^{1/2},\quad\forall~k\in[d].
    \end{align}
    \end{proposition}
    \begin{proof}
    If \(p_k(0)=0\), then \(p_k(t)\equiv 0\) for all \(t\ge0\). Thus it suffices to consider indices on the initial support. For \(k,m\) such that \(p_k(0),p_m(0)>0\), equation \eqref{one_part_p_ODE} yields
    \begin{align}\label{pkm_log}
        \frac{d}{dt}\log\frac{p_k(t)}{p_m(t)}
        =
        \frac{\dot p_k}{p_k}-\frac{\dot p_m}{p_m}
        =
        2(\lambda_k-\lambda_m).
    \end{align}
    Integrating \eqref{pkm_log}, we obtain
    \[
    p_k(t)
    =
    p_m(t)\frac{p_k(0)}{p_m(0)}e^{2(\lambda_k-\lambda_m)t}.
    \]
    Summing over \(k\) and using \(\sum_{k=1}^d p_k(t)=1\), we find
    \[
    1
    =
    p_m(t)\sum_{k=1}^d \frac{p_k(0)}{p_m(0)}e^{2(\lambda_k-\lambda_m)t},
    \]
    which implies the desired formula \eqref{pk_formula}. Next, since \eqref{one_part_ODE} has the form
    \[
    \dot c_k
    =
    c_k\left(\lambda_k-\sum_{l=1}^d \lambda_l c_l^2\right),
    \]
    the sign of each \(c_k(t)\) is preserved along the flow. Recalling that \(p_k=c_k^2\), we conclude that
    \[
    c_k(t)
    =
    \operatorname{sgn}(c_k(0))
    \left(
    \frac{c_k(0)^2 e^{2\lambda_k t}}{\sum_{l=1}^d c_l(0)^2 e^{2\lambda_l t}}
    \right)^{1/2},
    \quad\forall~ k\in[d].
    \]
    This proves \eqref{consensus_sol} and ends the proof.
    \end{proof}
    
 Using the explicit formula \eqref{consensus_sol} we shall show next that, in the consensus
    regime, only the largest eigenvalue on the initial support survives
    asymptotically.
    
    \begin{corollary}[Mode selection in the consensus regime]
    Let
    \[
    I_0:=\{k\in[d]: p_k(0)>0\},
    \quad
    \lambda_*:=\max_{k\in I_0}\lambda_k,
    \quad
    I_*:=\{k\in I_0:\lambda_k=\lambda_*\}.
    \]
    Then one gets
    \begin{align*}
        \lim_{t\to\infty}p_k(t)=
        \begin{dcases}
            0,\quad&\mbox{if}~~k\notin I_*,\\
            \frac{p_k(0)}{\sum_{j\in I_*}p_j(0)},&\mbox{if}~~k\in I_*.
        \end{dcases}
    \end{align*}
    \end{corollary}
    
    \begin{proof}
    If \(k\notin I_*\), then \(\lambda_k<\lambda_*\). By
    \eqref{pk_formula},
    \[
    p_k(t)
    =
    \frac{p_k(0)e^{2\lambda_k t}}
    {\sum_{l=1}^d p_l(0)e^{2\lambda_l t}}
    \le
    \frac{p_k(0)e^{2\lambda_k t}}
    {\sum_{j\in I_*} p_j(0)e^{2\lambda_* t}}
    =
    \frac{p_k(0)}{\sum_{j\in I_*} p_j(0)}
    e^{-2(\lambda_*-\lambda_k)t},
    \]
    hence \(p_k(t)\to0\). Now let \(k\in I_*\). Then, one has
    \begin{align*}
        p_k(t)&=\frac{p_k(0)e^{2\lambda_* t}}{\sum_{j\in I_*} p_j(0)e^{2\lambda_* t}+\sum_{l\notin I_*} p_l(0)e^{2\lambda_l t}}\\
        &=\frac{p_k(0)}{\sum_{j\in I_*} p_j(0)+\sum_{l\notin I_*} p_l(0)e^{-2(\lambda_*-\lambda_l)t}}
    \end{align*}
    Since \(\lambda_l<\lambda_*\) for \(l\notin I_*\), the second sum in the denominator tends to
    \(0\), and therefore
    \[
    p_k(t)\to
    \frac{p_k(0)}{\sum_{j\in I_*} p_j(0)}
    \]
    as $t\rightarrow \infty$, which finishes the proof.
    \end{proof}
    %============================================================================================================================

    \subsection{Balanced bipolar manifold}\label{subsec:bipolar_mfd}
    
    We next consider polarized configurations of the form
    \begin{align*}
    x_i(t)=s_i u(t), \quad s_i\in\{-1,1\},
    \end{align*}
    where the signs \(s_i\) encode a splitting of the population into two
    opposite groups. In general, this ansatz is not invariant. However, under a
    natural balance condition, the two groups contribute symmetrically to the
    softmax normalization, and the dynamics again reduces to a closed equation for
    a single profile \(u(t)\in\mathbb S^{d-1}\). This provides a second reduced
    regime, distinct from consensus, in which polarization rather than homogeneous
    alignment becomes the relevant organizing structure. Let
    \begin{align}\label{Sn_def}
        S_+:=\{i:s_i=1\},\quad S_-:=\{i:s_i=-1\},\quad n_+:=|S_+|,\quad n_-:=|S_-|.
    \end{align}
    If \(n_+=n_-\), then the two polarized groups carry equal weight in the
    softmax normalization, and the reduced dynamics closes. For this, we define the balanced bipolar manifold
    \[\mathcal M_{\mathrm{bbp}}:=\big\{X\in(\mathbb S^{d-1})^n~:~\exists~u\in\mathbb S^{d-1}~\mbox{such that}~x_i=s_iu,~\forall~i\in[n]~\mbox{with}~|S_+|=|S_-|\big\}.\]
    
    \begin{proposition}[Invariance of the balanced bipolar manifold $\mathcal M_{\mathrm{bbp}}$]
    The balanced bipolar manifold $\mathcal M_{\mathrm{bbp}}$ is invariant under the flow of \eqref{TF_ODE_sym}, i.e., there exists $u(t)\in\mathbb S^{d-1}$ such that
    \[
    x_i(t)=s_i u(t),\quad \forall\, t\ge0,\ i\in[n].
    \]
    Moreover, denoting
    \[
    u(t)=\sum_{k=1}^d u_k(t)e_k,
    \quad
    M(t):=\sum_{l=1}^d \lambda_l u_l^2(t),
    \]
    the coefficients satisfy
    \begin{align*}
        \dot u_k=u_k\,\alpha(M)\,(\lambda_k-M),\quad\alpha(M)=\tanh(\beta M),\quad\forall~t\ge0.
    \end{align*}
    \end{proposition}
    
    \begin{proof}
    Assume that \(x_i=s_i u\). Then, in the eigenbasis \(\{e_k\}_{k=1}^d\), the modal coefficients are given by
    \begin{align}\label{csu}
        c_{i,k}=s_i u_k,\quad\forall~i\in[n],\ k\in[d],
    \end{align}
    which leads to
    \[
    \sum_{l=1}^d \lambda_l c_{i,l}c_{m,l}
    =
    s_i s_m \sum_{l=1}^d \lambda_l u_l^2
    =
    s_i s_m M.
    \]
    It follows that
    \begin{align}\label{KC_eq}
    \begin{aligned}
        K_{im}(C)=
        \begin{dcases}
            \frac{e^{\beta s_m M}}{n_+e^{\beta M}+n_-e^{-\beta M}},\quad&\mbox{if}~~i\in S_+,\\
            \frac{e^{-\beta s_m M}}{n_+e^{-\beta M}+n_-e^{\beta M}}&\mbox{if}~~i\in S_-.
        \end{dcases}
    \end{aligned}
    \end{align}
    Now recall the modal equation \eqref{C_ODE}
    \[
    \dot c_{i,k}
    =
    \lambda_k\sum_{m=1}^n K_{im}(C)c_{m,k}
    -
    \left(\sum_{l=1}^d c_{i,l}\lambda_l\sum_{m=1}^nK_{im}(C)c_{m,l}\right)c_{i,k}.
    \]
    Substituting \eqref{csu} and \eqref{KC_eq}, we obtain
    \begin{align}\label{eq:balanced_cik}
    \dot c_{i,k}
    =
    u_k\left(\sum_{m=1}^nK_{im}(C)s_m\right)
    \left(\lambda_k-\sum_{l=1}^d\lambda_lu_l^2\right)
    =
    u_k\left(\sum_{m=1}^nK_{im}(C)s_m\right)(\lambda_k-M).
    \end{align}
    Thus the ansatz is invariant provided the factor \(\sum_{m=1}^nK_{im}(C)s_m\) is of the form \(s_i\) times a quantity independent of \(i\). A direct computation gives
    \begin{align*}
        \begin{aligned}
            \sum_{m=1}^nK_{im}(C)s_m=
            \begin{dcases}
                \frac{n_+e^{\beta M}-n_-e^{-\beta M}}{n_+e^{\beta M}+n_-e^{-\beta M}},\quad&\mbox{if}~~i\in S_+,\\
                \frac{n_+e^{-\beta M}-n_-e^{\beta M}}{n_+e^{-\beta M}+n_-e^{\beta M}},&\mbox{if}~~i\in S_-.
            \end{dcases}
        \end{aligned}
    \end{align*}
    Since \(n_+=n_-\) by assumption, this reduces to
    \[
    \sum_{m=1}^nK_{im}(C)s_m
    =
    \begin{cases}
    \tanh(\beta M), \quad&\mbox{if}~~ i\in S_+,\\[1mm]
    -\tanh(\beta M), &\mbox{if}~~ i\in S_-,
    \end{cases}
    \]
    or equivalently,
    \[
    \sum_{m=1}^nK_{im}(C)s_m=s_i\tanh(\beta M).
    \]
    Therefore \eqref{eq:balanced_cik} becomes
    \[
    \dot c_{i,k}
    =
    s_i u_k\tanh(\beta M)(\lambda_k-M).
    \]
    Using \(c_{i,k}=s_i u_k\), we conclude that
    \[
    \dot u_k=u_k\tanh(\beta M)(\lambda_k-M).
    \]
    This proves the claimed invariance and the reduced evolution equation.
    \end{proof}
    
    As in \Cref{subsec:css_mfd}, passing to the squared masses
    \[
    p_k:=u_k^2,
    \quad
    \sum_{k=1}^d p_k=1,
    \]
    we obtain the replicator-type equation
    \begin{align}\label{bipolar_p_ODE}
    \dot p_k
    =
    2p_k\,\alpha(M)\,(\lambda_k-M),
    \quad
    M=\sum_{l=1}^d \lambda_l p_l.
    \end{align}
    
    The long-time behavior of \eqref{bipolar_p_ODE} is governed by the weighted average $M$. The next lemma shows that \(M\)
    is monotone along the flow, and in particular that its sign is preserved.
    
    \begin{lemma}[Monotonicity of the weighted average $M$]\label{lem:M_monotone}
    Let \(p=(p_1,\cdots,p_d)\) solve \eqref{bipolar_p_ODE}. Then
    \[
    \dot M
    =
    2\alpha(M)\sum_{k=1}^d p_k(\lambda_k-M)^2.
    \]
    In particular, since \(\alpha(M)=\tanh(\beta M)\), we have
    \[
    M>0 \implies \dot M\ge0,
    \quad
    M<0 \implies \dot M\le0,
    \quad
    M=0 \implies \dot M=0.
    \]
    Hence the sign of \(M(t)\) is preserved along the flow.
    \end{lemma}
    
    \begin{proof}
    Using \eqref{bipolar_p_ODE}, we compute
    \[
    \dot M
    =
    \sum_{k=1}^d \lambda_k \dot p_k
    =
    2\alpha(M)\sum_{k=1}^d p_k\lambda_k(\lambda_k-M).
    \]
    Since
    \[
    \sum_{k=1}^d p_k(\lambda_k-M)
    =
    M-M\sum_{k=1}^d p_k
    =
    0,
    \]
    it follows that
    \begin{align*}
    \sum_{k=1}^d p_k\lambda_k(\lambda_k-M)
    &=
    \sum_{k=1}^d p_k\bigl((\lambda_k-M)+M\bigr)(\lambda_k-M)\\
    &=
    \sum_{k=1}^d p_k(\lambda_k-M)^2
    +
    M\sum_{k=1}^d p_k(\lambda_k-M)\\
    &=
    \sum_{k=1}^d p_k(\lambda_k-M)^2.
    \end{align*}
    Therefore,
    \[
    \dot M
    =
    2\alpha(M)\sum_{k=1}^d p_k(\lambda_k-M)^2.
    \]
    The sign conclusions follow immediately from
    \(\alpha(M)=\tanh(\beta M)\).
    \end{proof}
    
    This monotonicity determines the asymptotic mode-selection mechanism on the balanced bipolar manifold.
    
    \begin{proposition}[Asymptotic behavior of the reduced bipolar dynamics]\label{prop_asymp_bi_polar}
    Let \(p=(p_1,\cdots,p_d)\) solve \eqref{bipolar_p_ODE}, and define
    \[
    I_0:=\{k\in[d]:p_k(0)>0\},
    \quad
    \lambda_+:=\max_{k\in I_0}\lambda_k,
    \quad
    \lambda_-:=\min_{k\in I_0}\lambda_k.
    \]
    Set
    \[
    I_+:=\{k\in I_0:\lambda_k=\lambda_+\},
    \quad
    I_-:=\{k\in I_0:\lambda_k=\lambda_-\}.
    \]
    Then the following assertions hold:
    \begin{enumerate}
    \item If \(M(0)>0\), then
    \[
    p_k(t)\to0 \quad \text{for } k\notin I_+,
    \quad
    p_k(t)\to \frac{p_k(0)}{\sum_{j\in I_+}p_j(0)} \quad \text{for } k\in I_+,
    \]
    and, in particular,
    \[
    M(t)\to \lambda_+.
    \]
    
    \item If \(M(0)<0\), then
    \[
    p_k(t)\to0 \quad \text{for } k\notin I_-,
    \quad
    p_k(t)\to \frac{p_k(0)}{\sum_{j\in I_-}p_j(0)} \quad \text{for } k\in I_-,
    \]
    and, in particular,
    \[
    M(t)\to \lambda_-.
    \]
    
    \item If \(M(0)=0\), then \(p(t)\) is stationary:
    \[
    p_k(t)\equiv p_k(0),\quad \forall\, t\ge0,\ k\in[d].
    \]
    \end{enumerate}
    \end{proposition}
    \begin{proof}
    The support of \(p\) is preserved along the flow. Indeed, if \(p_k(t_0)=0\) for
    some \(t_0\), then \eqref{bipolar_p_ODE} gives \(\dot p_k(t_0)=0\), and uniqueness
    implies \(p_k(t)\equiv0\) for all \(t\ge t_0\). Hence the active set remains equal
    to \(I_0\). For \(k,\ell\in I_0\), \eqref{bipolar_p_ODE} yields
    \[
    \frac{d}{dt}\log\frac{p_k}{p_\ell}
    =
    \frac{\dot p_k}{p_k}-\frac{\dot p_\ell}{p_\ell}
    =
    2\alpha(M)(\lambda_k-\lambda_\ell).
    \]
    Therefore,
    \begin{align}\label{eq:ratio_formula}
    \frac{p_k(t)}{p_\ell(t)}
    =
    \frac{p_k(0)}{p_\ell(0)}
    \exp\!\left(2(\lambda_k-\lambda_\ell)\int_0^t \alpha(M(s))\,ds\right).
    \end{align}
    
    Assume first that \(M(0)>0\). By \Cref{lem:M_monotone}, \(M(t)>0\) for all \(t\ge0\),
    and \(M\) is nondecreasing. Since \(M(t)\) is a convex combination of
    \(\{\lambda_k:k\in I_0\}\), we have \(M(t)\le\lambda_+\). Hence \(M(t)\) converges to
    some \(M_\infty\in(0,\lambda_+]\). Moreover,
    \[
    \alpha(M(t))=\tanh(\beta M(t))\ge \tanh(\beta M(0))>0,
    \]
    so that
    \[
    \int_0^t \alpha(M(s))\,ds\to+\infty
    \quad\text{as }t\to\infty.
    \]
    Fix \(j\in I_+\). If \(k\notin I_+\), then \(\lambda_k-\lambda_j<0\), and
    \eqref{eq:ratio_formula} gives
    \[
    \frac{p_k(t)}{p_j(t)}\to0.
    \]
    It follows that \(p_k(t)\to0\) for all \(k\notin I_+\). On the other hand, if
    \(k,j\in I_+\), then \(\lambda_k=\lambda_j\), so \eqref{eq:ratio_formula} shows that
    \(p_k(t)/p_j(t)\) is constant in time. Since the mass outside \(I_+\) vanishes and
    \(\sum_{m=1}^d p_m(t)=1\), we obtain
    \[
    p_k(t)\to \frac{p_k(0)}{\sum_{m\in I_+}p_m(0)},
    \quad k\in I_+.
    \]
    Consequently,
    \[
    M(t)=\sum_{k=1}^d \lambda_k p_k(t)\to\lambda_+.
    \]
    
    The case \(M(0)<0\) is treated in the same way. By \Cref{lem:M_monotone},
    \(M(t)<0\) for all \(t\ge0\), and \(M\) is nonincreasing. Since \(M(t)\ge\lambda_-\),
    it follows that \(M(t)\to M_\infty\in[\lambda_-,0)\), while
    \[
    \alpha(M(t))=\tanh(\beta M(t))\le\tanh(\beta M(0))<0,
    \]
    so that
    \[
    \int_0^t \alpha(M(s))\,ds\to-\infty
    \quad\text{as }t\to\infty.
    \]
    Fixing \(j\in I_-\), we obtain from \eqref{eq:ratio_formula} that \(p_k(t)/p_j(t)\to0\)
    for every \(k\notin I_-\), since then \(\lambda_k-\lambda_j>0\). Hence
    \(p_k(t)\to0\) for \(k\notin I_-\). For \(k\in I_-\), the ratios \(p_k(t)/p_j(t)\)
    remain constant because \(\lambda_k=\lambda_j=\lambda_-\). Since the total mass on
    \(I_-\) therefore converges to \(1\), we conclude that
    \[
    p_k(t)\to \frac{p_k(0)}{\sum_{m\in I_-}p_m(0)},
    \quad k\in I_-,
    \]
    and hence
    \[
    M(t)\to\lambda_-.
    \]
    
    Finally, if \(M(0)=0\), then \Cref{lem:M_monotone} implies \(M(t)\equiv0\) for all
    \(t\ge0\). Therefore \(\alpha(M(t))=0\), and \eqref{bipolar_p_ODE} reduces to
    \(\dot p_k(t)=0\) for every \(k\in[d]\). Hence \(p(t)\) is stationary.
    \end{proof}
    
    The two reduced manifolds already exhibit two qualitatively different selection mechanisms. On the consensus manifold, the dynamics selects the largest eigenvalue on the initial support and leads to homogeneous alignment. On the balanced bipolar manifold, the sign of the weighted average \(M\) determines whether the dynamics is driven toward the largest or smallest eigenvalue on the initial support. In particular, if \(V\) is positive definite or negative definite, then this sign is fixed a priori, so Proposition \ref{prop_asymp_bi_polar} directly identifies the asymptotically selected extreme mode. These reduced dynamics motivate the analysis of pure-mode equilibria and global mode selection in the full system carried out in the subsequent sections.

%============================================================================================================================
%============================================================================================================================
    \section{Pure-mode equilibria and local stability}\label{sec_stab}

    Motivated by the reduced dynamics in \Cref{sec_reduced}, we now study equilibria of the full system supported on a single eigendirection of \(V\). These configurations are the natural candidates for long-time limits suggested by the mode-selection mechanisms identified above. Indeed, for any \(p\in[d]\) and any sign pattern \(s_1,\cdots,s_n\in\{\pm1\}\), the configuration
    \[
    x_i=s_i e_p,\quad\forall~i\in[n],
    \]
    is an equilibrium of \eqref{TF_ODE_sym}. In particular, concentration on a single eigendirection should be understood as concentration on the one-dimensional eigenspace \(\mathrm{span}\{e_p\}\), not necessarily as convergence to a single oriented state. This leads naturally to two geometrically distinct classes of pure-mode equilibria: \textit{homogeneous states} and \textit{sign-split states}.
    %============================================================================================================================    
    \subsection{Homogeneous pure states}\label{subsec:stab_homo}

    We consider the homogeneous pure state
    \[
    x_i^*=e_p,\quad\forall~i\in[n].
    \]
    For the local stability estimate, we first derive the linearization of the system \eqref{TF_ODE_sym} at the pure state. Since the dynamics evolves on the sphere, the linearization is taken in the tangent space at \(e_p\). Accordingly, only perturbations transverse to \(e_p\) are relevant. We therefore consider perturbations of the form
    \[
    x_i=\frac{e_p+y_i}{\|e_p+y_i\|},
    \quad
    \langle y_i,e_p\rangle=0.
    \]
    The homogeneous linearization is recovered from the general pure-mode linearization derived in \Cref{app_subsec:linear} by specializing to the case \(s_i\equiv 1\). In this case, one obtains
    \[
    \dot y_i=V\overline y-\lambda_p y_i,
    \quad
    \overline y:=\frac1n\sum_{j=1}^n y_j.
    \]
    Projecting onto the eigenbasis of \(V\), we obtain, for each \(k\neq p\),
    \begin{align}\label{eq:linearized_homogeneous}
    \dot y_{i,k}=\lambda_k\overline y_k-\lambda_p y_{i,k},
    \quad
    \overline y_k:=\frac1n\sum_{j=1}^n y_{j,k}.
    \end{align}
    Decomposing
    \[
    y_{i,k}=\overline y_k+\widetilde y_{i,k},
    \quad
    \sum_{i=1}^n \widetilde y_{i,k}=0,
    \]
    we obtain
    \begin{align}\label{mean_fluc_ode}
    \dot{\overline y}_k=(\lambda_k-\lambda_p)\overline y_k,
    \quad
    \dot{\widetilde y}_{i,k}=-\lambda_p \widetilde y_{i,k}.
    \end{align}
    Thus, for each transverse mode \(k\neq p\), the mean component evolves with rate
    \(\lambda_k-\lambda_p\), whereas the fluctuation components decay with rate \(-\lambda_p\).
    
    \begin{proposition}[Local stability of homogeneous pure states]\label{prop_lin_homo}
    Fix \(p\in[d]\). Then the homogeneous equilibrium
    \[
    x_i^*=e_p,\forall~i\in[n],
    \]
    is linearly asymptotically stable if
    \[
    \lambda_p>0
    \quad\text{and}\quad
    \lambda_k<\lambda_p \quad\text{for all } k\neq p.
    \]
    It is linearly unstable if either
    \[
    \lambda_p<0,
    \quad\text{or}\quad
    \lambda_k>\lambda_p \quad\text{for some } k\neq p.
    \]
    \end{proposition}
    
    \begin{proof}
    The linearized dynamics \eqref{eq:linearized_homogeneous} splits into the mean modes and fluctuation modes described in \eqref{mean_fluc_ode}. For each \(k\neq p\), the mean mode has growth rate \(\lambda_k-\lambda_p\), while the fluctuation modes have growth rate \(-\lambda_p\). Hence all linearized modes decay if and only if
    \[
    \lambda_p>0
    \quad\text{and}\quad
    \lambda_k<\lambda_p \quad\text{for all } k\neq p.
    \]
    If either \(\lambda_p<0\) or \(\lambda_k>\lambda_p\) for some \(k\neq p\), then the linearization admits a positive eigenvalue, and the equilibrium is linearly unstable.
    \end{proof}
    
    This proposition shows that stable homogeneous alignment can occur only at a positive-dominant mode. This naturally leads to the question whether a different stability mechanism may arise when the limiting configuration is concentrated on the same one-dimensional eigenspace but with opposite orientations.
    %============================================================================================================================
    \subsection{Sign-split pure states}\label{sign_split_linear}

    We next consider genuinely sign-split pure states of the form
    \begin{align}\label{eq:sign_split_equilibrium}
        x_i^*=s_i e_p,
        \quad
        s_i\in\{\pm1\},
    \end{align}
    where the sign pattern is assumed to be nonconstant, so that both signs \(+1\) and \(-1\) occur. Unlike the homogeneous case in \Cref{subsec:stab_homo}, these equilibria allow concentration on the same one-dimensional eigenspace with opposite orientations. With the notation \(S_\pm\) and \(n_\pm\) from \eqref{Sn_def}, we now study the corresponding linearized dynamics.
    
    To study local stability, we consider tangent perturbations around \eqref{eq:sign_split_equilibrium} in the same form as before:
    \[x_i=\frac{x_i^*+y_i}{\|x_i^*+y_i\|},\quad\langle y_i,x_i^*\rangle=0.\]
    The general pure-mode linearization derived in \Cref{app_subsec:linear} yields the tangent system
    \begin{align}\label{y_split}
        \dot y_i=-\gamma_i y_i+\sum_{j=1}^n K_{ij}^* V y_j,
    \end{align}
    where
    \[
    K_{ij}^*:=K_{ij}(C^*)
    \quad\text{and}\quad
    \gamma_i:=\lambda_p s_i\sum_{j=1}^n K_{ij}^* s_j.
    \]
    
    At the equilibrium, the coefficients \(K_{ij}^*\) are constant on each block determined by the partition \(S_+\cup S_-\). Denote
    \begin{align*}
        a_+ &= \frac{e^{\beta\lambda_p}}{n_+e^{\beta\lambda_p}+n_-e^{-\beta\lambda_p}},
        &
        b_+ &= \frac{e^{-\beta\lambda_p}}{n_+e^{\beta\lambda_p}+n_-e^{-\beta\lambda_p}},\\
        a_- &= \frac{e^{\beta\lambda_p}}{n_-e^{\beta\lambda_p}+n_+e^{-\beta\lambda_p}},
        &
        b_- &= \frac{e^{-\beta\lambda_p}}{n_-e^{\beta\lambda_p}+n_+e^{-\beta\lambda_p}},
    \end{align*}
    and
    \[
    \gamma_+ = \lambda_p(n_+a_+-n_-b_+),
    \quad
    \gamma_- = \lambda_p(n_-a_- - n_+b_-).
    \]
    
    Since the equilibrium is supported on the eigendirection \(e_p\), the transverse modes \(k\neq p\) decouple in the eigenbasis of \(V\). For each \(k\neq p\), the corresponding perturbation variables satisfy
    \begin{align}\label{eq:linearized_sign_split}
    \dot y_{i,k}=
    \begin{dcases}
        -\gamma_+\,y_{i,k}
        +\lambda_k\Bigl(a_+\sum_{j\in S_+}y_{j,k}+b_+\sum_{j\in S_-}y_{j,k}\Bigr),
        & \text{if } i\in S_+,\\[2mm]
        -\gamma_-\,y_{i,k}
        +\lambda_k\Bigl(b_-\sum_{j\in S_+}y_{j,k}+a_-\sum_{j\in S_-}y_{j,k}\Bigr),
        & \text{if } i\in S_-.
    \end{dcases}
    \end{align}
    The derivation of \eqref{y_split}--\eqref{eq:linearized_sign_split} is deferred to \Cref{app_subsec:linear}.
    
    Introducing the group averages
    \[
    \overline y_{+,k}:=\frac1{n_+}\sum_{i\in S_+}y_{i,k},
    \quad
    \overline y_{-,k}:=\frac1{n_-}\sum_{i\in S_-}y_{i,k},
    \]
    and the fluctuations
    \[
    \widetilde y_{i,k}:=y_{i,k}-\overline y_{+,k}\quad (i\in S_+),
    \quad
    \widetilde y_{i,k}:=y_{i,k}-\overline y_{-,k}\quad (i\in S_-),
    \]
    we obtain
    \begin{align}\label{tildey_dot}
    	\dot{\widetilde y}_{i,k}=-\gamma_+\widetilde y_{i,k},\quad(i\in S_+),\quad\dot{\widetilde y}_{i,k}=-\gamma_-\widetilde y_{i,k},\quad(i\in S_-).
    \end{align}
    On the two-dimensional mean subspace spanned by \((\overline y_{+,k},\overline y_{-,k})\), the reduced dynamics is
    \begin{align}\label{bary_dot}
    	\frac{d}{dt}
	\begin{pmatrix}
		\overline y_{+,k}\\[1mm]
		\overline y_{-,k}
	\end{pmatrix}
	B_k
	\begin{pmatrix}
		\overline y_{+,k}\\[1mm]
		\overline y_{-,k}
	\end{pmatrix},\quad
	B_k=
        \begin{pmatrix}
            \lambda_k n_+ a_+ - \gamma_+ & \lambda_k n_- b_+\\
            \lambda_k n_+ b_- & \lambda_k n_- a_- - \gamma_-
        \end{pmatrix}.
    \end{align}
    
    Thus, for each transverse mode \(k\neq p\), the linearized operator splits into three invariant parts: fluctuations within \(S_+\), fluctuations within \(S_-\), and a two-dimensional mean component coupling the two sign groups. This immediately yields the following stability criterion.
    
    \begin{theorem}[Local stability of sign-split equilibria]\label{stability_split}
    	The equilibrium \eqref{eq:sign_split_equilibrium} is linearly asymptotically stable if and only if
    	\[
    	\gamma_+>0,
    	\quad
    	\gamma_->0,
    	\]
    	and, for every \(k\neq p\),
    	\[
    	\operatorname{tr}(B_k)<0,
    		\quad
    		\det(B_k)>0.
    	\]
    \end{theorem}

    \begin{proof}[Sketch of proof]
    For each transverse mode \(k\neq p\), the linearized system admits the invariant decomposition described in \Cref{app_subsec:split_decomp}, and \Cref{app_subsec:thm52} identifies the spectrum on each invariant component. The claimed criterion then follows by requiring all scalar modes to be stable and imposing linear asymptotic stability on the two-dimensional mean component associated with the block \(B_k\). This argument is carried out in detail in \Cref{app_subsec:thm52}.
    \end{proof}

    \begin{remark}[$\beta$-dependent stability regimes]\label{rem:beta_dependent_stability}
        Set
        \begin{align*}
            c_\beta(\lambda_p):=e^{2\beta\lambda_p},\quad r:=\frac{n_+}{n_-}.
        \end{align*}
        Then the stability conditions in \Cref{stability_split} are equivalent to the following.

        For $\lambda_p>0$,
        \begin{align*}
            \lambda_p>\frac{1}{2\beta}|\ln r|,\quad\lambda_k<\lambda_p\sigma(c_\beta,r),\quad\forall~k\ne p.
        \end{align*}

        For $\lambda_p<0$,
        \begin{align*}
            \lambda_p<-\frac{1}{2\beta}|\ln r|,\quad\lambda_p<\lambda_k<\lambda_p\sigma(c_\beta,r),\quad\forall~k\ne p.
        \end{align*}

        Here,
        \begin{align*}
            \sigma(c_\beta,r)&:=\frac{(c_\beta-r)(c_\beta r-1)}{r(c_\beta^2-1)}.
        \end{align*}
        In particular, in the admissible regimes above one has
        \[
            0<\sigma(c_\beta,r)<1
            \quad\text{if } \lambda_p>0,
            \quad
            -1<\sigma(c_\beta,r)<0
            \quad\text{if } \lambda_p<0.
        \]
        Moreover, since
        \[
            \sigma(c,r)=-\sigma(c^{-1},r),
        \]
        the cases $\lambda_p=a$ and $\lambda_p=-a$ yield the same upper bound for $\lambda_k$.

        A representative slice of this \(\beta\)-dependent stability boundary will be visualized in \Cref{sec_numer}.
    \end{remark}

    The analysis above makes precise the dichotomy suggested by the reduced dynamics in \Cref{sec_reduced}. Whereas the consensus dynamics favors homogeneous alignment, the balanced bipolar dynamics points to a polarized mode-selection mechanism. The local stability analysis in this section shows how this distinction is reflected at the level of pure-mode equilibria of the full system. This distinction will reappear in the global mode-selection analysis of the next section.
%============================================================================================================================
%============================================================================================================================
    \section{Global mode selection}\label{sec_global}
    
    \Cref{sec_reduced} and \Cref{sec_stab} reveal two distinct mechanisms of mode selection in the symmetric self-attention dynamics \eqref{TF_ODE_sym}. Dominant positive modes favor homogeneous alignment, whereas negative modes may instead support sign-split configurations. In this section, we show how these reduced and local pictures extend to global selection results under additional structural assumptions.

    We treat two regimes. First, in the positive-dominant regime, namely when
    \[\lambda_1>\max_{k\ge2}|\lambda_k|,\]
    we prove, under a one-sided cone assumption, convergence of the full \(n\)-particle system to the dominant positive eigendirection. Second, in the two-particle negative-definite regime, namely when $n=2$ and
    \[0>\lambda_1\ge\lambda_2\ge\cdots\ge\lambda_d,\]
    we show that the dynamics becomes asymptotically sign-split and that, when $\lambda_d$ is simple, the selected limit is the sign-split state associated with the smallest eigenvalue.
%============================================================================================================================

    \subsection{Global selection under a dominant-positive mode of \(V\)}

    Assume that one positive eigendirection dominates all others in modulus:
    \[
    \lambda_1>\max_{k\ge2}|\lambda_k|.
    \]
    Under this condition, we shall prove below a uniform one-sided lower bound on the first modal coordinate defines a forward-invariant cone. This excludes drift toward competing eigendirections and yields exponential decay of all transverse modes. Consequently, the full configuration converges to the homogeneous state aligned with \(e_1\).
    
    \begin{theorem}[Selection of the dominant positive eigendirection in a forward-invariant cone]
    	\label{thm:cone_selection_e1}
    	Assume that
    	\[
    	\lambda_1>\max_{k\ge2}|\lambda_k|,
    	\]
        and let \(C=(c_{i,k})\) be a solution of the self-attention dynamics \eqref{C_ODE}. Suppose that there exists \(\delta>0\) such that
    	\[
    	c_{i,1}(0)\ge \delta
    	\quad\forall~i\in[n].
    	\]
    	Then the following assertions hold:
    	\begin{enumerate}
    		\item The cone
    		\[
    		\mathcal C_\delta:=\Bigl\{(x_i)_{i=1}^n\in(\mathbb S^{d-1})^n:\ c_{i,1}\ge \delta,\quad\forall~i\in[n]\Bigr\}
    		\]
    		is forward invariant.
    		
    		\item For every \(k\neq1\), if we define
    		\[
    		r_{i,k}(t):=\frac{c_{i,k}(t)}{c_{i,1}(t)},
    		\quad
    		R_k(t):=\max_{1\le i\le n}|r_{i,k}(t)|,
    		\]
    		then
    		\[
    		R_k(t)\le R_k(0)e^{-\delta(\lambda_1-|\lambda_k|)t},
    		\quad t\ge0.
    		\]
    		
    		\item For every $i\in[n]$,
    		\[
    		x_i(t)\to e_1
    		\quad\text{as }t\to\infty.
    		\]
    	\end{enumerate}
    \end{theorem}
    \begin{proof}[Sketch of proof]
    	The argument has three ingredients. First, the assumption \(c_{i,1}(0)\ge\delta\) defines a forward-invariant cone: the minimum of the first modal coordinate is nondecreasing because the nonlinear term \(\phi_i\) is uniformly bounded above by \(\lambda_1\). Second, once \(c_{i,1}\) stays uniformly positive, one can introduce the transverse ratios
    \[
    r_{i,k}=\frac{c_{i,k}}{c_{i,1}},\quad k\neq1,
    \]
    for which the nonlinear terms cancel exactly. The resulting ratio system satisfies a comparison estimate in the second assertion. Finally, this exponential decay forces every transverse mode \(c_{i,k}(t)\), \(k\neq1\), to vanish, and the positivity of \(c_{i,1}\) selects the orientation \(+e_1\). The detailed estimates are given in \Cref{proof_thm61}.
    \end{proof}
    
    \Cref{thm:cone_selection_e1} immediately yields the corresponding statement for the opposite orientation. As the proof is analogous to \Cref{thm:cone_selection_e1}, we omit the proof.
    
    \begin{corollary}[Selection of the opposite orientation]\label{cor:selection_e1}
    	Under the assumptions of \Cref{thm:cone_selection_e1}, if instead
    	\[
    	c_{i,1}(0)\le -\delta<0
    	\quad\forall~i\in[n],
    	\]
    	then
    	\[
    	x_i(t)\to -e_1
    	\quad\forall~i\in[n].
    	\]
    \end{corollary}
    
    The cone argument above captures a genuinely alignment-driven regime. Once the leading positive mode is uniformly present across all particles, it remains dominant and determines the global asymptotic state. This mechanism is inherently one-sided, however, and does not apply to the polarized configurations arising in the negative-definite case. To analyze that regime, we instead exploit the geometry of anti-alignment, which becomes particularly transparent in the two-particle system.
    %============================================================================================================================
    \subsection{Two-particle sign-split selection in the negative-definite regime}\label{subsec:two}

    We now turn to the two-particle case \eqref{TF_ODE_sym}, i.e., $n=2$, and assume that \(V\) is symmetric and negative definite
    \[0>\lambda_1\ge\lambda_2\ge\cdots\ge\lambda_d.\]
    In this regime, the relevant geometry is no longer alignment but anti-alignment. The key quantity is the pairwise correlation
    \[
    \rho(t):=\langle x_1(t),x_2(t)\rangle.
    \]
    The next lemma shows that \(\rho\) is strictly decreasing whenever the two particles are neither fully aligned nor fully anti-aligned.

    \begin{lemma}[Strict monotonicity of the pairwise correlation]\label{rho_monotone_two_particle}
    Suppose that the symmetric matrix \(V\) is negative definite and let \((x_1(t),x_2(t))\in\mathbb S^{d-1}\times\mathbb S^{d-1}\) be a solution to the two-particle system \eqref{TF_ODE_sym}. Then, $\rho(t)$ is strictly decreasing on $\rho(t)\in(-1,1)$.
    \end{lemma}
    \begin{proof}[Sketch of proof]
	Writing the dynamics in terms of the sum and difference variables
    \[
    s:=x_1+x_2,
    \quad
    d:=x_1-x_2,
    \]
    and using the positive definite matrix \(B:=-V\), one obtains an explicit expression for \(\dot\rho\) in terms of quadratic forms in \(s\) and \(d\) and the corresponding softmax weights. After a suitable rearrangement, this expression is seen to be strictly negative whenever \(-1<\rho<1\). The detailed computation is given in \Cref{proof_lem63}.
    \end{proof}

    Hence \(\rho\) is strictly decreasing on every interval on which \(\rho(t)\in(-1,1)\).
    In particular, every nontrivial trajectory satisfies \(\rho(t)\to -1\), and is therefore
    driven asymptotically toward the two-particle sign-split manifold
    \[
        \mathcal M_{\mathrm{bbp}}
        :=
        \{(u,-u):u\in \mathbb S^{d-1}\},
    \]
    which coincides with the balanced bipolar manifold introduced in \Cref{subsec:bipolar_mfd}
    when \(n=2\).
    
    Once the dynamics approaches \(\mathcal M_{\mathrm{bbp}}\), the remaining question is which
    eigendirection is selected within this polarized manifold. The next theorem shows that,
    outside a measure-zero exceptional set, the selected mode is the one associated with the
    smallest eigenvalue.

    \begin{theorem}[Almost-everywhere selection of the most negative eigendirection for two particles]
    \label{ae_selection_most_negative}
    Assume the hypotheses of \Cref{rho_monotone_two_particle}, and suppose that $V$ has its smallest eigenvalue $\lambda_d$ simple. Let \((x_1(t),x_2(t))\) be the solution to \eqref{TF_ODE_sym} with arbitrary initial data \((x_1(0),x_2(0))\in \mathbb S^{d-1}\times \mathbb S^{d-1}\). Then for almost every initial datum $(x_1(0),x_2(0))$, there exists \(\sigma\in\{\pm1\}\) such that
    \[
    (x_1(t),x_2(t))\to (\sigma e_d,-\sigma e_d)
    \quad\text{as }~~t\to\infty.
    \]
    \end{theorem}
    \begin{proof}[Sketch of proof]
        By \Cref{rho_monotone_two_particle}, every nontrivial trajectory satisfies \(\rho(t)\to -1\), so every \(\omega\)-limit set is contained in \(\mathcal M_{\mathrm{bbp}}\). On \(\mathcal M_{\mathrm{bbp}}\), the dynamics reduces to the polarized system from \Cref{sec_reduced}. Moreover, in the two-particle setting, we introduce the Lyapunov function
        \begin{align*}
            L(u):=\langle u,Vu\rangle-\lambda_d,
        \end{align*}
        This is a shifted version of the quantity \(M\) appearing in \Cref{lem:M_monotone}. Since \(V\) is negative definite, \(M\) is monotone decreasing along the reduced sign-split dynamics, and therefore so is \(L\). The shift by \(\lambda_d\) makes  \(L\) nonnegative, with \(L(u)=0\) precisely on the eigenspace corresponding to the smallest eigenvalue \(\lambda_d\). A LaSalle-type argument then shows that every \(\omega\)-limit set of the reduced dynamics is contained in the largest invariant subset of \(\{\dot L=0\}\), which here coincides with the union of the sign-split equilibrium manifolds associated with the eigenspaces of \(V\).
    
        If the selected eigenspace is the one corresponding to the smallest eigenvalue \(\lambda_d\), then the simplicity of \(\lambda_d\) implies convergence to one of the two states \((e_d,-e_d)\) or \((-e_d,e_d)\). It therefore remains to exclude convergence to the higher-eigenvalue equilibrium manifolds for generic initial data. This is done by combining the sign-split linearization from \Cref{sign_split_linear} with the stable manifold theorem: every equilibrium manifold associated with an eigenvalue \(\lambda>\lambda_d\) is normally hyperbolic and possesses a nontrivial unstable direction, so its stable set has positive codimension and therefore measure zero. The detailed argument is given in \Cref{proof_thm65}.
    \end{proof}

Thus the only nongeneric alternatives are trajectories converging to sign-split equilibria associated with higher eigendirections. These exceptional initial data are contained in the union of the corresponding stable manifolds and hence form a measure-zero set.

Taken together, the results of this section provide a global counterpart to the reduced and local analyses developed earlier. In the positive-dominant regime, together with the one-sided cone condition, the dynamics selects a homogeneous pure mode through a forward-invariant cone mechanism. In the two-particle negative-definite regime, the dynamics first approaches the sign-split manifold and, outside a measure-zero exceptional set, selects the pure sign-split state associated with the smallest eigenvalue.
%============================================================================================================================
%============================================================================================================================
    \section{Numerical regimes of mode selection}\label{sec_numer}
    In this section, we complement the analytical results of \Cref{sec_global} with numerical experiments. Our goal is not only to illustrate the two mode-selection mechanisms proved above, but also to probe nearby regimes that are not covered by the present theory. Accordingly, we consider theorem-validation simulations, exploratory simulations, and a parameter-threshold visualization associated with the local stability theory. In the positive-dominant setting, we compare the one-sided cone regime from \Cref{thm:cone_selection_e1} and \Cref{cor:selection_e1} with initial configurations that do not satisfy the one-sided sign assumption. In the negative-definite setting, we first illustrate the two-particle dynamics described in \Cref{rho_monotone_two_particle,ae_selection_most_negative}, and then briefly examine the multi-particle regime \(n\geq 3\), where the asymptotic selection mechanism remains open. Finally, we return to the local stability criterion for sign-split pure states and visualize how the admissible stability window depends on \(\beta\) and on the population ratio \(r=n_+/n_-\).
%============================================================================================================================
    \subsection{Positive-dominant regime: one-sided and mixed-sign initial data}

    We first consider the positive-dominant regime
    \[
    \lambda_1>\max_{k\ge2}|\lambda_k|.
    \]
    In this regime, our analysis predicts selection of the dominant eigendirection under a one-sided sign condition on the first mode. More precisely, if the initial configuration satisfies
    \[
    c_{i,1}(0)\ge \delta>0, \quad \forall~i\in[n],~\mbox{for some }\delta\in(0,1),
    \]
    then \Cref{thm:cone_selection_e1} yields convergence to the homogeneous state \(e_1\).
    
    To compare these rigorous predictions with the particle dynamics, we consider two classes of initial data. The first consists of one-sided configurations satisfying the cone condition above. The second consists of mixed-sign initial data in the first mode, for which the current cone argument does not apply. This second class therefore probes the behavior beyond the proved one-sided regime.

    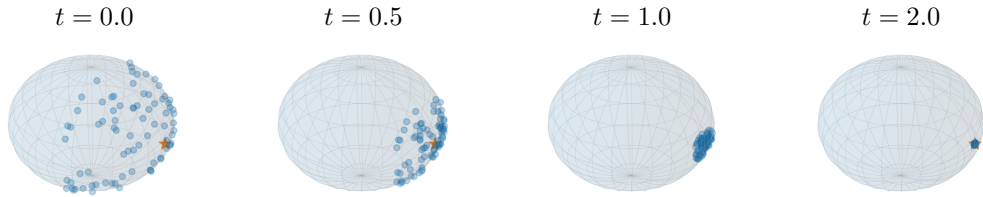
\begin{figure}
        \centering
        \begin{overpic}
            [scale=0.43]{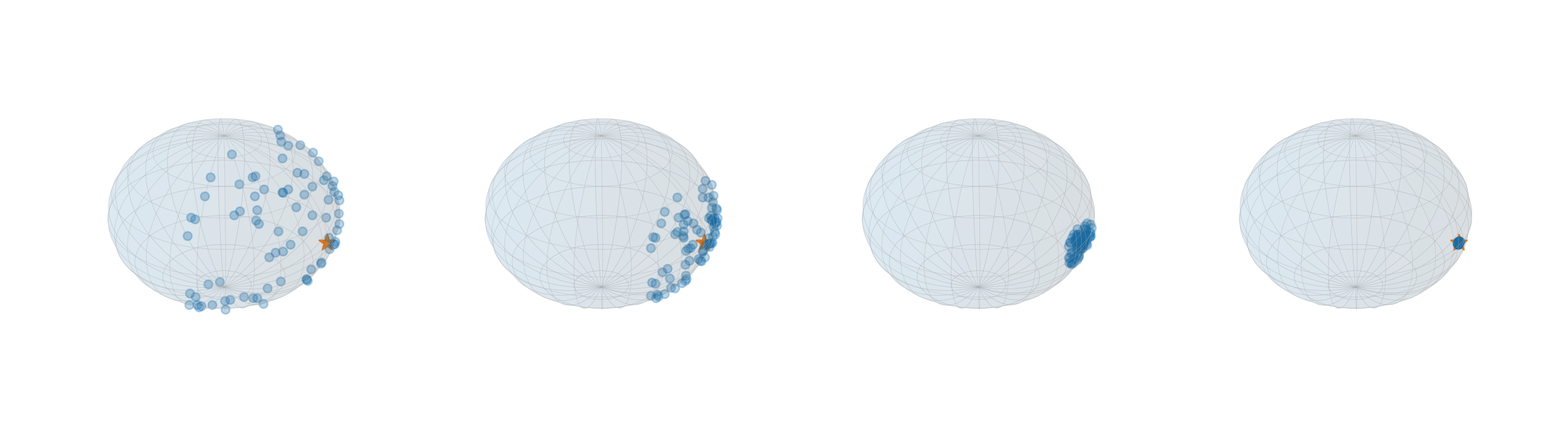}
            \put(11,23){\small$t=0.0$}
            \put(35,23){\small$t=0.5$}
            \put(59,23){\small$t=1.0$}
            \put(83,23){\small$t=2.0$}
        \end{overpic}

        \vspace{-1cm}
        
        \caption{Snapshots of the particle system in the positive-dominant regime at times \(t=0,0.5,1.0,2.0\). The initial data satisfy the one-sided condition \(c_{i,1}(0)\ge \delta>0\) for all \(i\in[n]\). The orange star indicates the dominant eigendirection \(e_1\) associated with the largest eigenvalue \(\lambda_1\). Here \(n=80\) and \(\beta=1\).}
        \label{fig:positive_snapshot}
    \end{figure}    
    We first examine the one-sided case. \Cref{fig:positive_snapshot} shows a representative trajectory. Starting from a dispersed configuration contained in the positive cone determined by the dominant eigendirection, the particles rapidly contract and form a single cluster near \(e_1\), marked by the orange star. This provides a direct numerical illustration of the alignment mechanism predicted by \Cref{thm:cone_selection_e1}.
    
    \begin{figure}
        \centering
        \begin{subfigure}{0.32\textwidth}
            \centering
            \begin{overpic}[width=\linewidth]{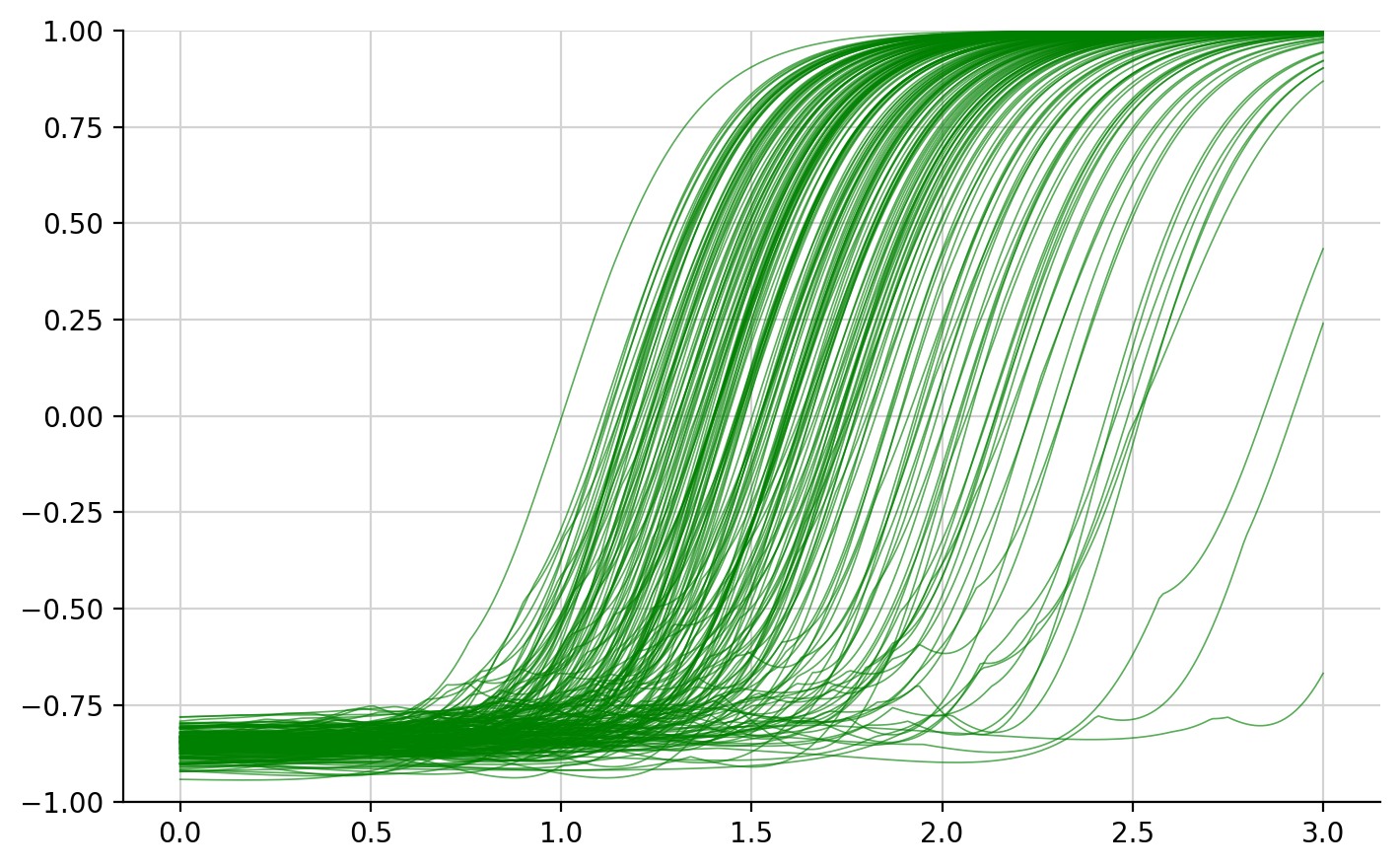}
            \put(47,-5){\small time}
            \put(-5,25){\rotatebox{90}{\small$\rho_{\mathrm{min}}$}}
            \put(44,65){\small $\beta=0.1$}
            \end{overpic}
        \end{subfigure}
        \begin{subfigure}{0.32\textwidth}
            \centering
            \begin{overpic}[width=\linewidth]{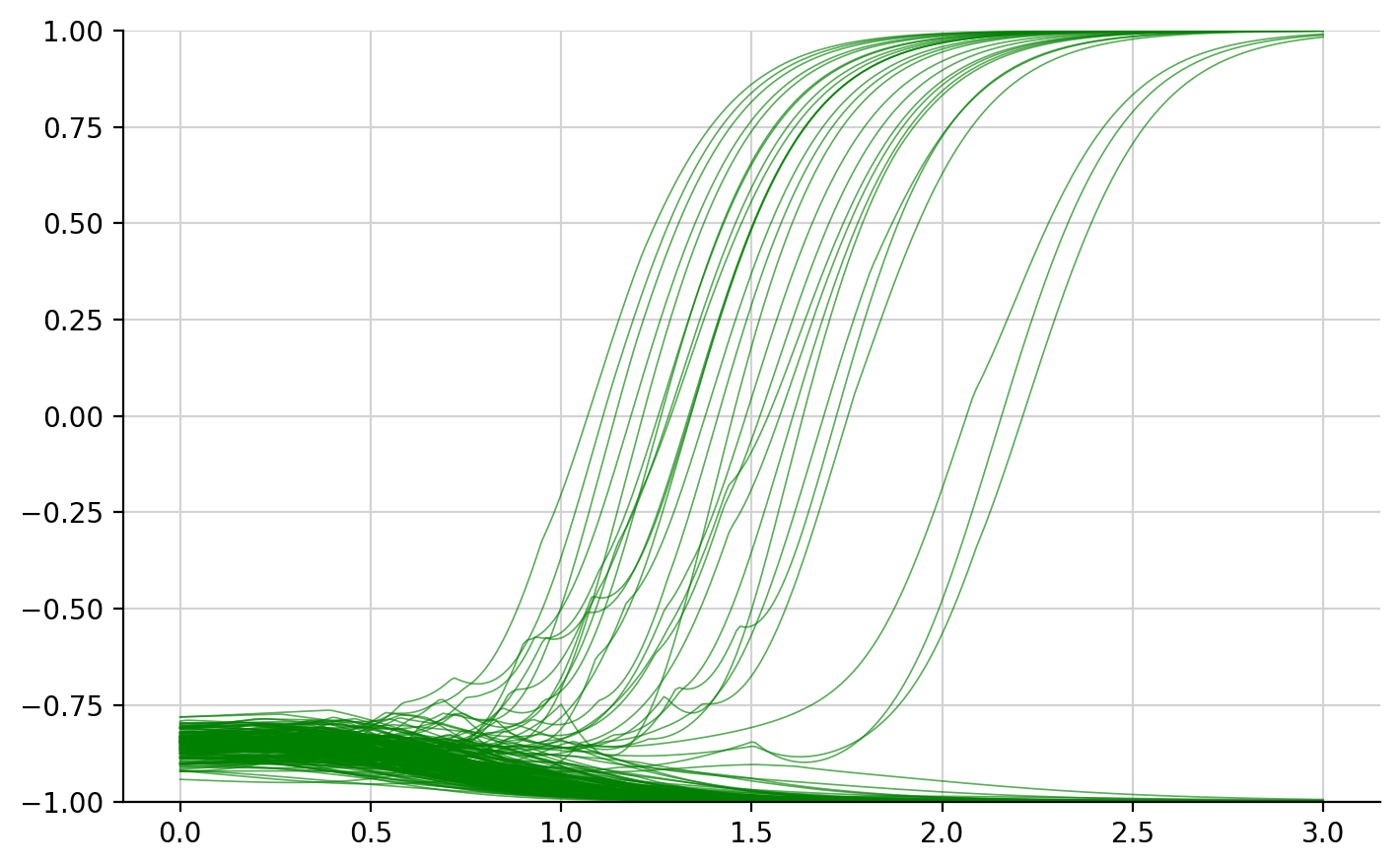}
            \put(44,65){\small $\beta=1.0$}
            \end{overpic}
        \end{subfigure}
        \begin{subfigure}{0.32\textwidth}
            \centering
            \begin{overpic}[width=\linewidth]{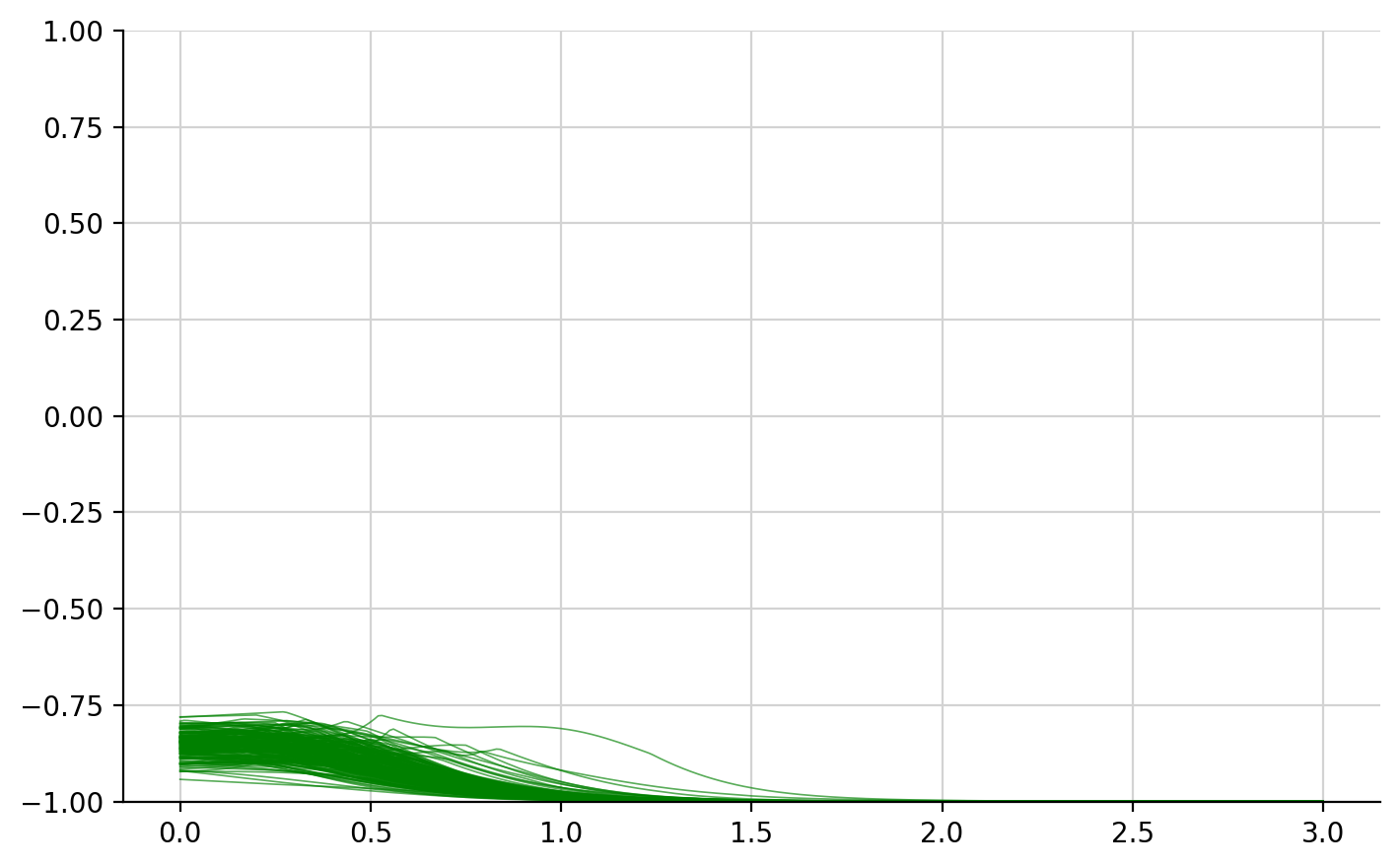}
            \put(44,65){\small $\beta=1.5$}
            \end{overpic}
        \end{subfigure}
        \caption{Time evolution of \(\rho_{\mathrm{min}}(t):=\min_{i,j\in[n]}\langle x_i(t),x_j(t)\rangle\) for a fixed matrix \(V\) with dominant positive eigenvalue \(\lambda_1\). The values $\rho_{\mathrm{min}}=1$ and $-1$ correspond to homogeneous alignment and to the presence of at least one antipodal pair, respectively. Each curve corresponds to one trial, and the same collection of initial data is used across the three panels. The particle number is \(n=80\), and each panel contains \(200\) trials.}
        \label{fig:positive_beta}
    \end{figure}
    We next turn to mixed-sign initial data and quantify the outcome through
    \[
    \rho_{\mathrm{min}}(t):=\min_{i,j\in[n]}\langle x_i(t),x_j(t)\rangle.
    \]
    If \(\rho_{\mathrm{min}}(t)\to 1\), then all pairwise inner products approach \(1\), indicating convergence toward consensus. In contrast, trajectories with \(\rho_{\mathrm{min}}(t)\) approaching \(-1\) exhibit a pronounced antipodal separation and therefore signal a sign-split-type behavior.
    
    To isolate the effect of the attention sharpness parameter, we fix the matrix \(V\) and reuse the same \(200\) initial configurations for all three values \(\beta=0.1,1.0,1.5\). The resulting trajectories of \(\rho_{\mathrm{min}}\) are shown in green in \Cref{fig:positive_beta}. For small attention sharpness \(\beta=0.1\), most trials evolve toward consensus. For the intermediate value \(\beta=1.0\), both outcomes are observed: many runs still approach consensus, but a non-negligible portion develop strong antipodal separation. For the larger value \(\beta=1.5\), the dynamics are largely dominated by polarization, with most trajectories driving \(\rho_{\mathrm{min}}(t)\) toward \(-1\). This example shows that, outside the one-sided cone regime, the observed long-time behavior depends strongly on \(\beta\) even when the spectral condition \(\lambda_1>\max_{k\ge2}|\lambda_k|\) is fixed.
%============================================================================================================================
    \subsection{Two-particle negative-definite regime: sign-split selection}

We next consider the two-particle case \(n=2\) under the negative-definite assumption
\[
    0>\lambda_1> \lambda_2> \cdots > \lambda_d .
\]
In this regime, \Cref{rho_monotone_two_particle} shows that the pairwise correlation
\[
    \rho(t):=\langle x_1(t),x_2(t)\rangle
\]
is strictly decreasing whenever \(-1<\rho(t)<1\). Accordingly, nontrivial trajectories are
driven toward the sign-split manifold
\[
    M_{\rm bbp}:=\{(u,-u):u\in \mathbb S^{d-1}\}.
\]
To illustrate this behavior numerically, we sample random initial conditions on
\(\mathbb S^{d-1}\times \mathbb S^{d-1}\) and evolve the two-particle system for a fixed
negative-definite diagonal matrix \(V\).

To track the approach to \(M_{\rm bbp}\), we monitor the correlation \(\rho(t)\). To identify the selected eigendirection, we also consider the averaged modal mass associated with the smallest eigenvalue,
\[
    m_d(t):=\frac{1}{2}\bigl(c_{1,d}(t)^2+c_{2,d}(t)^2\bigr).
\]
If the trajectory converges as in \Cref{ae_selection_most_negative}, namely
\[
    (x_1(t),x_2(t))\to (\sigma e_d,-\sigma e_d)
    \quad\text{for some }\sigma\in\{\pm1\},
\]
then necessarily \(\rho(t)\to -1\) and \(m_d(t)\to 1\).

The pairwise correlation \(\rho\) is plotted in blue in \Cref{fig:negative_two_particle}, and the observed monotone decay is consistent with \Cref{rho_monotone_two_particle,ae_selection_most_negative}. In \Cref{fig:negative_rho}, the sampled trajectories exhibit the monotone decay of \(\rho(t)\) toward \(-1\), illustrating the approach to the sign-split manifold. In \Cref{fig:negative_md}, the corresponding modal mass \(m_d(t)\) tends to \(1\), indicating selection of the eigendirection associated with the smallest eigenvalue. Taken together, these two observables provide a numerical illustration of two-particle sign-split selection mechanism developed in \Cref{subsec:two}.

\begin{figure}[t]
    \centering
    \begin{subfigure}{0.48\textwidth}
        \centering
        \begin{overpic}
            [scale=0.37]{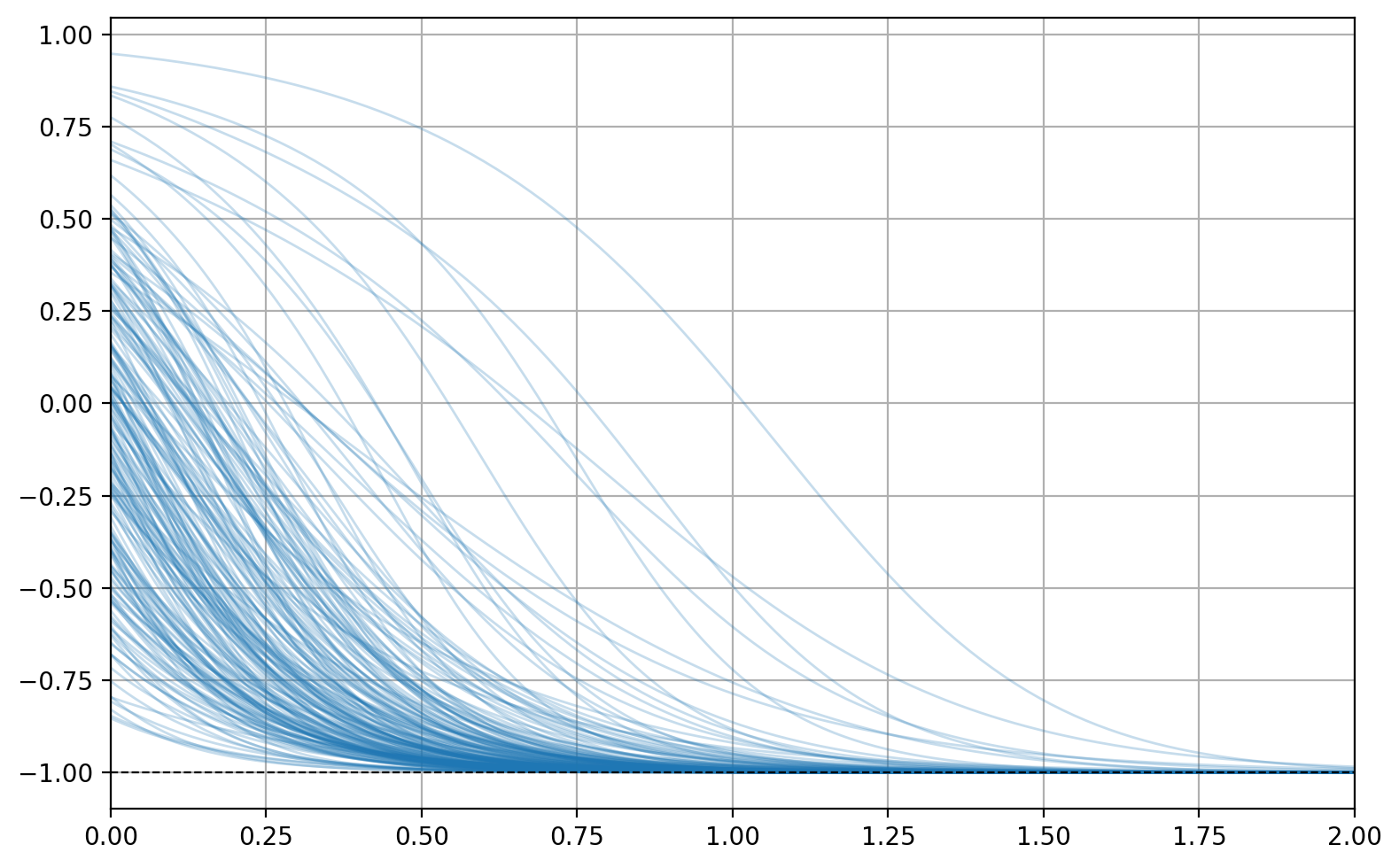}
            \put(47,-3){\small time}
            \put(-4,30){\rotatebox{90}{\small$\rho$}}
        \end{overpic}
        \caption{
        Time evolution of the pairwise correlation.
        }
        \label{fig:negative_rho}
    \end{subfigure}
    \hfill
    \begin{subfigure}{0.48\textwidth}
        \centering
        \begin{overpic}
            [scale=0.37]{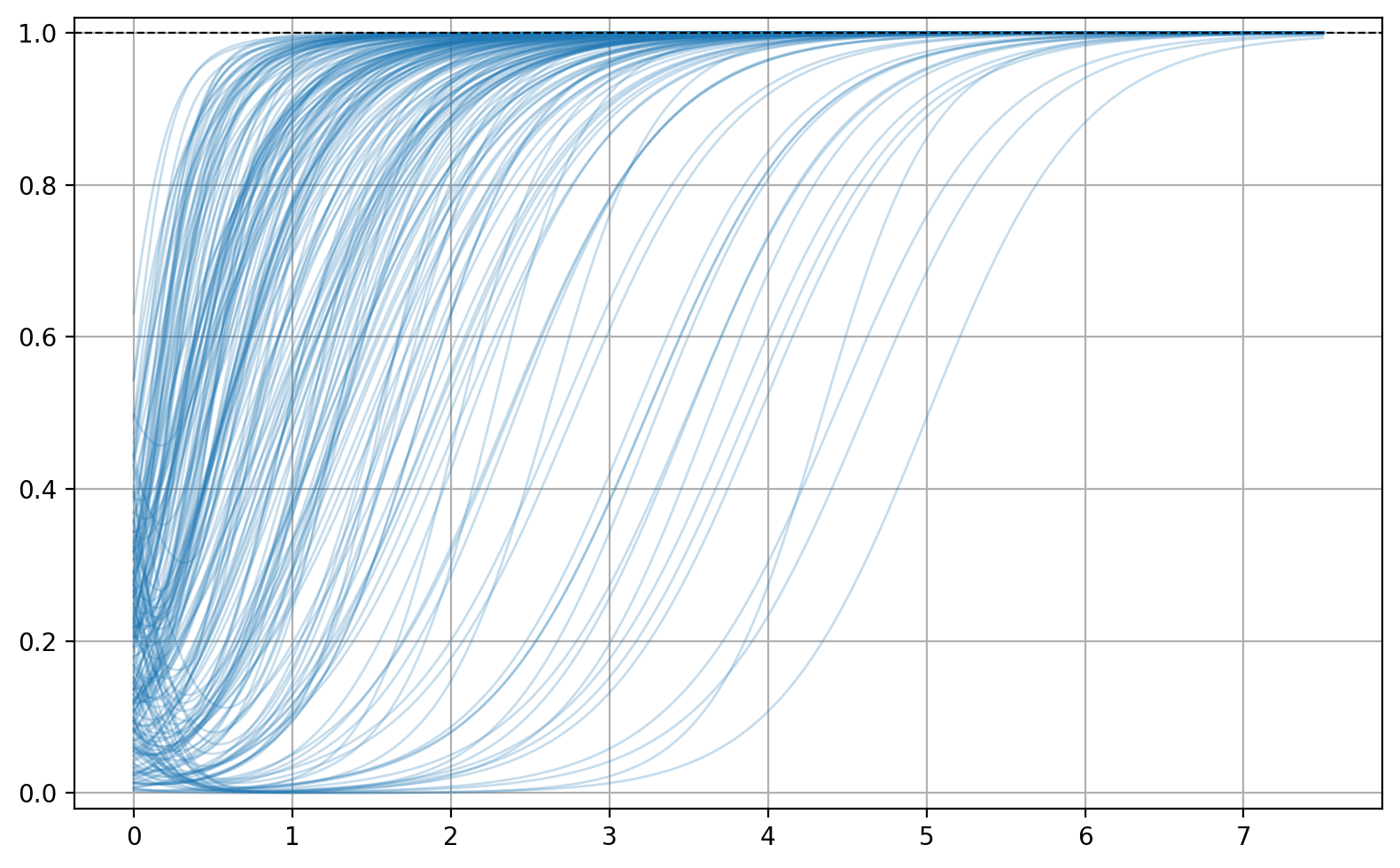}
            \put(47,-3){\small time}
            \put(-4,27){\rotatebox{90}{\small$m_d$}}
        \end{overpic}
        \caption{
        Time evolution of the averaged modal mass.
        }
        \label{fig:negative_md}
    \end{subfigure}
    \caption{
    Numerical illustration of the two-particle negative-definite regime.
    (a) shows the decay of the pairwise correlation toward \(-1\). (b) shows the concentration of the modal mass
    on the most negative eigendirection. In both panels, each curve corresponds to one random initial condition.
    }
    \label{fig:negative_two_particle}
\end{figure}
    \subsection{Negative-definite regime beyond two particles}

Here, we explore the negative-definite regime beyond the two-particle analysis in \Cref{subsec:two} by considering systems with \(n\ge 3\). Unlike the two-particle case, the global geometry is no longer described by a single pairwise correlation, since more than two particles cannot all be mutually antipodal. Thus, the monotonicity mechanism in \Cref{rho_monotone_two_particle} does not directly yield a global description for larger populations.

We consider a fixed negative-definite diagonal matrix \(V\) with
\[
    0>\lambda_1> \lambda_2> \cdots > \lambda_d,
\]
and sample random initial configurations on \((\mathbb S^{d-1})^n\). To monitor the
geometry of the population, we use the pairwise correlation observables
\[
    \rho_{\min}(t):=\min_{i<j}\langle x_i(t),x_j(t)\rangle,\quad\rho_{\max}(t):=\max_{i<j}\langle x_i(t),x_j(t)\rangle,
\]
and
\[
    \rho_{\mathrm{abs}}(t):=\frac{2}{n(n-1)}\sum_{i<j}|\langle x_i(t),x_j(t)\rangle|.
\]
To identify the selected spectral direction, we also track the averaged modal masses
\[
    m_k(t):=\frac{1}{n}\sum_{i=1}^n c_{i,k}(t)^2,\quad\forall~k\in[d].
\]

As shown in Figure~\ref{fig:negative_multi}, the numerical behavior suggests a possible multi-particle analogue of the two-particle picture from Lemma~\ref{rho_monotone_two_particle} and Theorem~\ref{ae_selection_most_negative}. In particular, the pairwise observables indicate that the population approaches an approximately bipolar configuration: particles concentrate near two antipodal directions, while same-group pairs remain positively correlated. At the same time, the modal masses show concentration on the eigendirection associated with the smallest eigenvalue. Thus, in this experiment, the negative-definite multi-particle system exhibits bipolar collapse along the most negative eigendirection. This numerical observation suggests that the sign-split selection mechanism proved in the two-particle case may persist for larger populations. However, in the absence of a global selection theorem for \(n\ge 3\), this should be regarded as exploratory evidence rather than a rigorous conclusion.

\begin{figure}[t]
    \centering
    \begin{subfigure}{0.48\textwidth}
        \centering
        \begin{overpic}
            [scale=0.37]{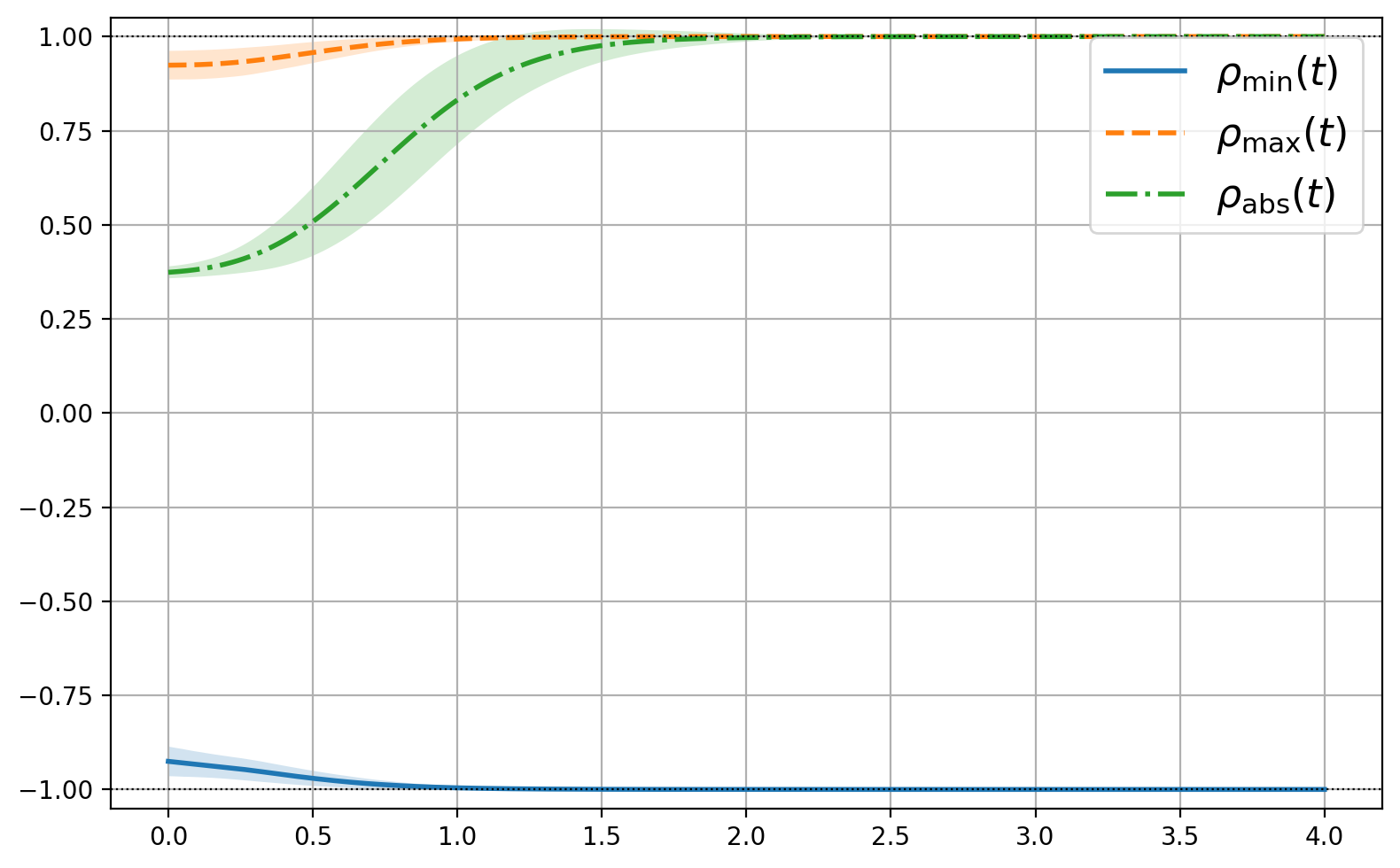}
            \put(47,-1.5){\small time}
        \end{overpic}
        \caption{
        Pairwise correlation observables.
        }
        \label{fig:negative_multi_pairwise}
    \end{subfigure}
    \hfill
    \begin{subfigure}{0.48\textwidth}
        \centering
        \begin{overpic}
            [scale=0.37]{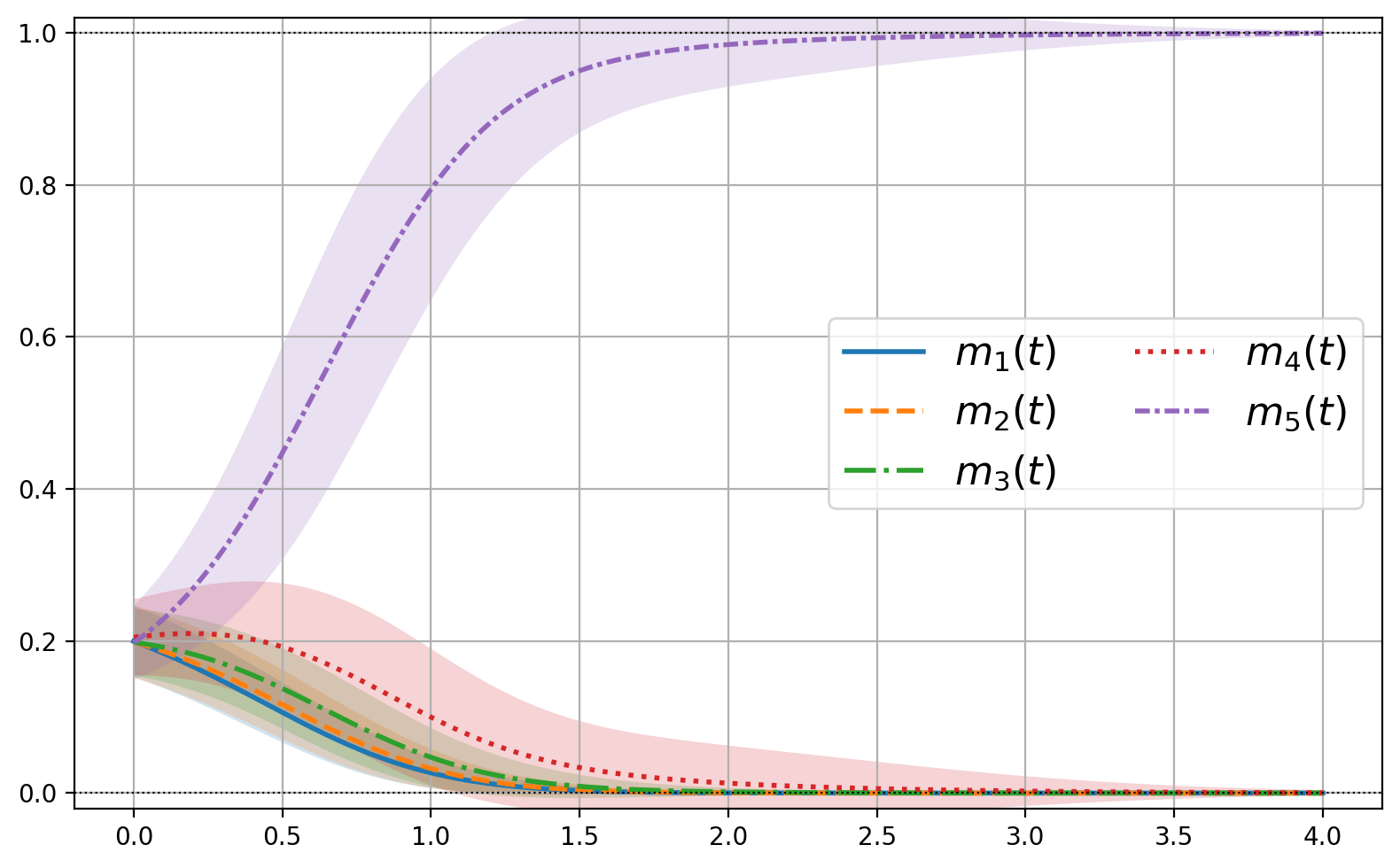}
            \put(47,-1.5){\small time}
        \end{overpic}
        \caption{
        Averaged modal masses.
        }
        \label{fig:negative_multi_modal_masses}
    \end{subfigure}
    \caption{
    Monte Carlo observables for the negative-definite regime with $n=20$ on $5$-dimensional space. The curves show empirical means over random initial configurations, and the shaded regions indicate one standard deviation. (a) shows that the population approaches an approximately bipolar configuration, as reflected by \(\rho_{\min}(t)\), \(\rho_{\max}(t)\), and \(\rho_{\mathrm{abs}}(t)\). (b) shows that the averaged modal mass concentrates on the eigendirection corresponding to the smallest eigenvalue \(\lambda_d\).
    }
    \label{fig:negative_multi}
\end{figure}

    \subsection{Stability thresholds for sign-split pure states}
    We now return to the local stability theory of sign-split pure states and visualize the
\(\beta\)-dependent boundary at which such polarized states become linearly admissible.
The purpose of this subsection is only to illustrate the threshold in
\Cref{rem:beta_dependent_stability}, rather than to provide a full parameter study.
For this reason, we focus on the representative case \(\lambda_p=1\).

In this case, the upper stability threshold for the transverse eigenvalues is
\[
    \lambda_p \sigma(c_\beta,r)
    =
    \sigma(e^{2\beta},r),
\]
which we plot as a function of the attention sharpness parameter \(\beta\), for several
values of \(r=n_+/n_-\). According to \Cref{rem:beta_dependent_stability}, the
sign-split equilibrium supported on \(e_p\) can be linearly stable only in the admissible
regime
\[
    \beta>\frac{1}{2}|\ln r|.
\]
At the endpoint \(\beta=\frac{1}{2}|\ln r|\), one has
\(\sigma(e^{2\beta},r)=0\). As \(\beta\) increases beyond this endpoint, the upper
stability threshold becomes positive and increases from \(0\).

The same curve can also be used to read the upper stability threshold for the case
\(\lambda_p=-1\). In that case, however, the admissible interval is bounded below by
\(-1\), so that
\[
    -1<\lambda_k<\sigma(e^{2\beta},r).
\]
For general values of \(\lambda_p\), the threshold changes nonlinearly with \(\lambda_p\),
since \(c_\beta=e^{2\beta\lambda_p}\).
    
    Figure~\ref{fig:beta_threshold} shows this threshold for several choices of $r$.
    The empty circles mark the points
    \[
        \beta=\frac12|\ln r|,
        \quad
        \lambda_p\sigma(c_\beta,r)=0,
    \]
    where the stability window first becomes nonempty. As $\beta$ increases, the upper
    bound increases and approaches $\lambda_p=1$. Thus, for larger attention sharpness, the
    local stability condition allows a wider range of transverse eigenvalues $\lambda_k$ below
    the selected positive mode. The dependence on $r$ reflects the imbalance between the two
    sign groups: the more unbalanced the split is, the larger the value of $\beta$ required before
    the sign-split state can become stable.
    
    \begin{figure}[t]
        \centering
        \begin{overpic}
            [scale=0.5]{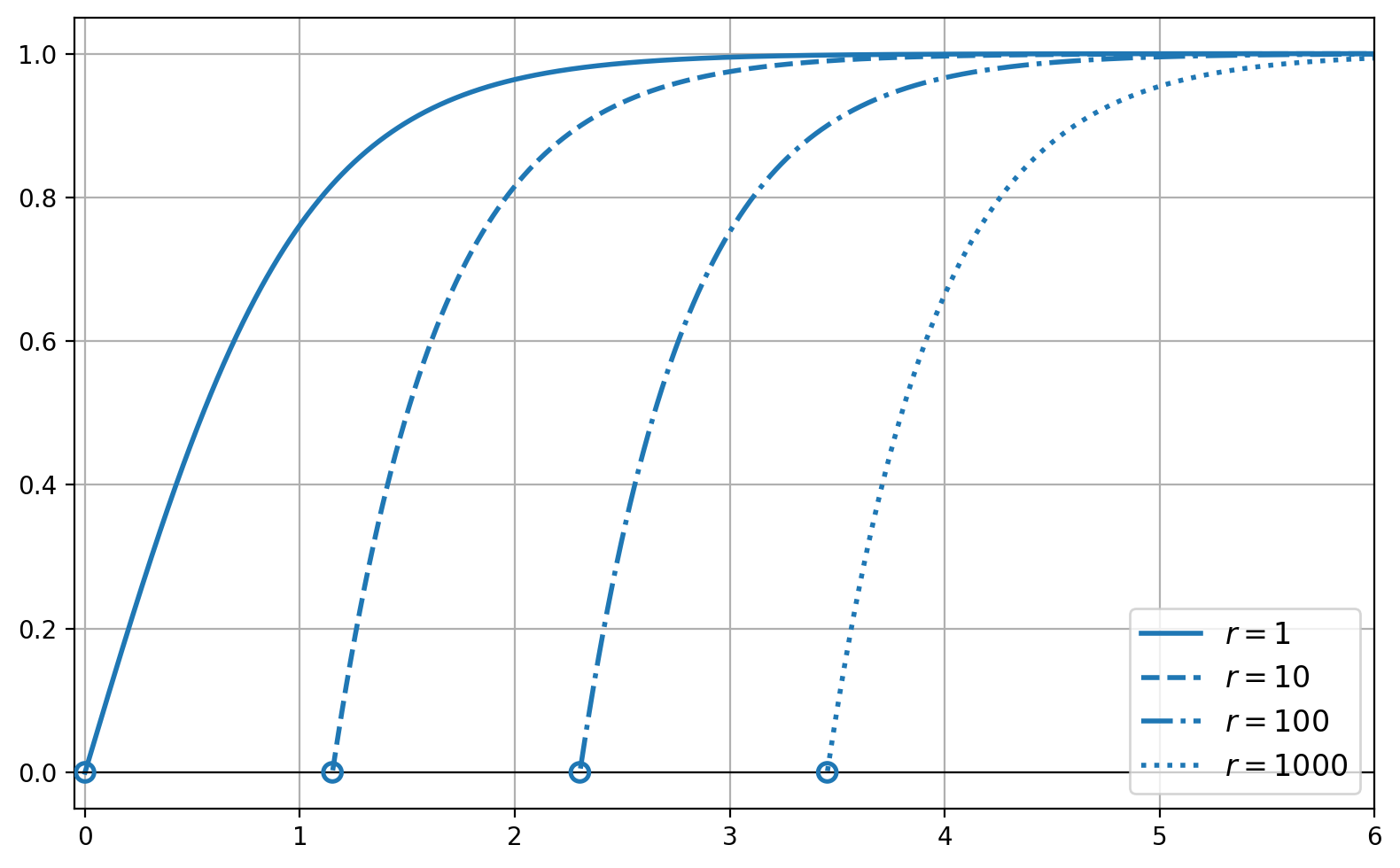}
            \put(47,-3){\small time}
            \put(-4,24){\rotatebox{90}{\small$\lambda_p\sigma(c_\beta,r)$}}
        \end{overpic}
        \caption{
        Representative \(\beta\)-dependent stability boundary for sign-split pure states,
shown for \(\lambda_p=1\). The plotted curve is the upper bound
\(\lambda_p\sigma(c_\beta,r)=\sigma(e^{2\beta},r)\). Different line styles correspond to
        different population ratios $r=n_+/n_-$. The empty circles indicate the threshold
        points $\beta=\frac12|\ln r|$, where the admissible stability regime begins. For \(r=1\), the threshold point at \(\beta=0\) is interpreted as a limiting value.
        }
        \label{fig:beta_threshold}
    \end{figure}
%============================================================================================================================
    \section{Conclusion}\label{sec_conclusion}
    In this paper, we studied a symmetric finite-particle self-attention dynamics on the sphere from the viewpoint of spectral mode selection. Under the assumption
    \[Q^\top K = V = V^\top,\]
    the system admits both a variational structure and an exact modal reformulation, which together make it possible to analyze how the spectrum of \(V\) governs the long-time behavior of the flow. At the level of reduced dynamics, the consensus and balanced bipolar manifolds reveal two qualitatively different selection mechanisms, corresponding respectively to homogeneous alignment and sign-split polarization. These mechanisms are reflected in the full system through the local stability theory of pure-mode equilibria and the global selection results proved here in the positive-dominant and two-particle negative-definite regimes.

    From a broader perspective, the present work provides a rigorous finite-particle analysis of a Transformer-inspired self-attention flow in a symmetric setting. Rather than addressing training dynamics or architectural expressivity, our focus has been on the intrinsic nonlinear dynamics generated by attention-type interactions and on the asymptotic patterns selected by the underlying spectrum. In this sense, the symmetric model studied here offers a mathematically tractable baseline in which mechanisms of alignment, polarization, and spectral competition can be understood explicitly.

    A natural next step is to investigate asymmetric perturbations that break the exact self-adjoint structure underlying the present analysis. From this viewpoint, an important question is which parts of the mode-selection picture established here remain robust under small asymmetry, and how the transition from the symmetric regime to genuinely non-symmetric self-attention dynamics alters the alignment and polarization mechanisms identified in this paper. It would also be of interest to understand whether the finite-particle spectral mechanisms found here persist for larger particle systems or admit meaningful counterparts in related mean-field descriptions. The explicit \(\beta\)- and imbalance-dependent thresholds obtained for sign-split equilibria also suggest a possible diagnostic framework for detecting transitions between alignment-dominated and polarization-dominated regimes. We hope that the present work can serve as a useful starting point for future studies of self-attention dynamics beyond the symmetric setting.
%============================================================================================================================
%============================================================================================================================
    \section*{Acknowledgements}
    J. Yoon would like to thank the Alexander von Humboldt Stiftung for support via a postdoctoral research fellowship.

%============================================================================================================================
%============================================================================================================================

%============================================================================================================================%============================================================================================================================
    \appendix
    \section{Proofs for \Cref{sec_stab}}\label{app_proof_5}

This appendix contains the technical proofs omitted from \Cref{sec_stab}.

%============================================================================================================================

\subsection{Linearization around pure-mode equilibria}\label{app_subsec:linear}

Here, we derive the linearized systems used in \Cref{subsec:stab_homo} and \Cref{sign_split_linear}.
We work around a general pure-mode equilibrium
\[
x_i^*=s_i e_p,
\quad s_i\in\{\pm1\},
\quad i\in[n].
\]
The homogeneous case in \Cref{subsec:stab_homo} corresponds to the special choice
\(s_i\equiv 1\), while the genuinely sign-split case in \Cref{sign_split_linear}
corresponds to a nonconstant sign pattern.

\begin{lemma}[Linearization around a pure-mode equilibrium]\label{lem:app_pure_mode_linearization}
Let
\[
x_i^*=s_i e_p,
\quad s_i\in\{\pm1\},
\]
be a pure-mode equilibrium of \eqref{TF_ODE_sym}. Consider perturbations of the form
\begin{align}\label{x_pert}
    x_i=\frac{x_i^*+y_i}{\|x_i^*+y_i\|},\quad \langle y_i,x_i^*\rangle=0,
\end{align}
and set
\[
|y|:=\max_{1\le m\le n}\|y_m\|.
\]
Then the linearized tangent system is
\begin{align}\label{eq:app_general_pure_linearization}
\dot y_i
=
-\gamma_i y_i+\sum_{j=1}^n K_{ij}^* V y_j,
\quad i\in[n],
\end{align}
where
\[
K_{ij}^*:=K_{ij}(x^*),
\quad
\gamma_i:=\lambda_p s_i\sum_{j=1}^n K_{ij}^* s_j.
\]
\end{lemma}

\begin{proof}
From \eqref{x_pert}, we have
\[
x_i=x_i^*+y_i+O(|y|^2).
\]
Moreover, since \(x_i^*=\pm e_p\), the tangent condition implies that \(y_i\) has no
\(e_p\)-component.

Recall that the vector field is given by
\[
F_i(x)=P^\perp_{x_i}\left(\sum_{j=1}^n K_{ij}(x)Vx_j\right),
\quad
K_{ij}(x):=
\frac{\exp\bigl(\beta\langle x_i,Vx_j\rangle\bigr)}
{\sum_{\ell=1}^n \exp\bigl(\beta\langle x_i,Vx_\ell\rangle\bigr)}.
\]
As in \Cref{sec_prelim}, we write
\[
\mathcal A_i(x):=\sum_{j=1}^n K_{ij}(x)Vx_j,
\]
so that
\[
F_i(x)=\mathcal A_i(x)-\langle x_i,\mathcal A_i(x)\rangle x_i.
\]

We first expand the softmax coefficients. Since
\[
Vx_j^*=\lambda_p s_j e_p,
\]
one has
\[
\langle x_i,Vx_j\rangle
=
\langle x_i^*+y_i,\;V(x_j^*+y_j)\rangle+O(|y|^2).
\]
Using \(y_i\perp e_p\), \(y_j\perp e_p\), and the symmetry of \(V\), we obtain
\[
\langle y_i,Vx_j^*\rangle
=
\lambda_p s_j\langle y_i,e_p\rangle
=0,
\]
and
\[
\langle x_i^*,Vy_j\rangle
=
s_i\langle e_p,Vy_j\rangle
=
s_i\langle Ve_p,y_j\rangle
=
\lambda_p s_i\langle e_p,y_j\rangle
=0.
\]
Therefore,
\[
\langle x_i,Vx_j\rangle
=
\langle x_i^*,Vx_j^*\rangle+O(|y|^2)
=
\lambda_p s_i s_j+O(|y|^2).
\]
Hence
\[
K_{ij}(x)=K_{ij}^*+O(|y|^2),
\quad
K_{ij}^*:=K_{ij}(x^*).
\]

Next, using again \(Vx_j^*=\lambda_p s_j e_p\), we find
\[
\mathcal A_i(x)
=
\sum_{j=1}^n K_{ij}^*\bigl(\lambda_p s_j e_p+Vy_j\bigr)+O(|y|^2)
=
\lambda_p\left(\sum_{j=1}^n K_{ij}^*s_j\right)e_p
+
\sum_{j=1}^n K_{ij}^*Vy_j
+
O(|y|^2).
\]
It follows that
\begin{align*}
\langle x_i,\mathcal A_i(x)\rangle
&=
\Bigl\langle s_i e_p+y_i,\,
\lambda_p\left(\sum_{j=1}^n K_{ij}^*s_j\right)e_p
+
\sum_{j=1}^n K_{ij}^*Vy_j
\Bigr\rangle
+
O(|y|^2)\\
&=
\lambda_p s_i\sum_{j=1}^n K_{ij}^*s_j
+
O(|y|^2),
\end{align*}
where we used again \(y_i\perp e_p\) and \(\langle e_p,Vy_j\rangle=0\).
Thus, setting
\[
\gamma_i:=\lambda_p s_i\sum_{j=1}^n K_{ij}^*s_j,
\]
we obtain
\[
F_i(x)
=
-\gamma_i y_i+\sum_{j=1}^n K_{ij}^*Vy_j+O(|y|^2).
\]
This yields the linearized tangent system \eqref{eq:app_general_pure_linearization}.
\end{proof}

We now record the two specializations used in \Cref{sec_stab}.

\paragraph{Homogeneous pure states (\Cref{subsec:stab_homo}).}
Assume \(s_i\equiv 1\), so that \(x_i^*=e_p\) for all \(i\in[n]\). Then
\[
K_{ij}^*=\frac1n,
\quad
\gamma_i=\lambda_p,
\]
and \eqref{eq:app_general_pure_linearization} reduces to
\[
\dot y_i
=
-\lambda_p y_i+\frac1n\sum_{j=1}^n V y_j
=
V\bar y-\lambda_p y_i,
\quad
\bar y:=\frac1n\sum_{j=1}^n y_j.
\]
Expanding
\[
y_i=\sum_{k\ne p} y_{i,k} e_k,
\quad
\bar y=\sum_{k\ne p} y_k e_k,
\quad
y_k:=\frac1n\sum_{j=1}^n y_{j,k},
\]
we obtain
\[
\dot y_{i,k}=\lambda_k y_k-\lambda_p y_{i,k},
\quad k\ne p,
\]
which is precisely \eqref{eq:linearized_homogeneous}. Decomposing
\[
y_{i,k}=y_k+\widetilde y_{i,k},
\quad
\sum_{i=1}^n \widetilde y_{i,k}=0,
\]
yields
\[
\dot y_k=(\lambda_k-\lambda_p)y_k,
\quad
\dot{\widetilde y}_{i,k}=-\lambda_p\widetilde y_{i,k},
\]
that is, \eqref{mean_fluc_ode}.

\paragraph{Sign-split pure states (\Cref{sign_split_linear}).}
Assume that the sign pattern is nonconstant, and use the notation \(S_\pm\), \(n_\pm\),
\(a_\pm\), \(b_\pm\), and \(\gamma_\pm\) introduced in \Cref{sign_split_linear}. Since
\[
\langle x_i^*,Vx_j^*\rangle=\lambda_p s_i s_j,
\]
the coefficients \(K_{ij}^*\) are constant on the blocks determined by \(S_+\cup S_-\), namely
\[
K_{ij}^*=
\begin{dcases}
a_+, & i\in S_+,\ j\in S_+,\\
b_+, & i\in S_+,\ j\in S_-,\\
b_-, & i\in S_-,\ j\in S_+,\\
a_-, & i\in S_-,\ j\in S_-.
\end{dcases}
\]
Since each perturbation is tangent to the sphere at \(x_i^*=\pm e_p\), we may write
\[
y_i=\sum_{k\ne p} y_{i,k}e_k.
\]
Substituting this expansion into \eqref{eq:app_general_pure_linearization} and using
\(Ve_k=\lambda_k e_k\), we obtain, for each fixed \(k\ne p\),
\begin{align}\label{eq:app_mode_system_sign_split}
\dot y_{i,k}
=
-\gamma_i y_{i,k}
+\lambda_k\sum_{j=1}^n K_{ij}^* y_{j,k}.
\end{align}
Equivalently,
\[
\dot y_{i,k}=
\begin{dcases}
-\gamma_+\,y_{i,k}
+\lambda_k\Bigl(a_+\sum_{j\in S_+}y_{j,k}+b_+\sum_{j\in S_-}y_{j,k}\Bigr),
& i\in S_+,\\[2mm]
-\gamma_-\,y_{i,k}
+\lambda_k\Bigl(b_-\sum_{j\in S_+}y_{j,k}+a_-\sum_{j\in S_-}y_{j,k}\Bigr),
& i\in S_-,
\end{dcases}
\]
which is precisely \eqref{eq:linearized_sign_split}.

Introducing the group averages and fluctuations as in \Cref{sign_split_linear}, one obtains
\eqref{tildey_dot} by subtracting the corresponding group mean on each sign group, while
\eqref{bary_dot} follows by averaging \eqref{eq:linearized_sign_split} over \(S_+\) and \(S_-\),
respectively.

%============================================================================================================================

\subsection{Invariant decomposition for the sign-split linearization}\label{app_subsec:split_decomp}

In this subsection, we fix \(k\ne p\) and consider the modewise linear operator
\(L_k:\mathbb R^n\to\mathbb R^n\) associated with \eqref{eq:app_mode_system_sign_split},
defined by
\[
(L_k z)_i
=
-\gamma_i z_i+\lambda_k\sum_{j=1}^n K_{ij}^* z_j,
\quad \forall~i\in[n].
\]

\begin{lemma}[Invariant decomposition]\label{prop:app_invariant_decomposition}
The space \(\mathbb R^n\) decomposes as
\[
\mathbb R^n = V_+\oplus V_-\oplus V_{\mathrm{mean}},
\]
where
\begin{align*}
V_+
&:=\left\{z\in\mathbb R^n:
z_i=0 \text{ for } i\in S_-,
\ \sum_{i\in S_+}z_i=0\right\},\\
V_-
&:=\left\{z\in\mathbb R^n:
z_i=0 \text{ for } i\in S_+,
\ \sum_{i\in S_-}z_i=0\right\},\\
V_{\mathrm{mean}}
&:=\left\{z\in\mathbb R^n:
z_i=u_+ \text{ on } S_+,\ z_i=u_- \text{ on } S_-
\text{ for some } u_\pm\in\mathbb R\right\}.
\end{align*}
Each of these subspaces is invariant under \(L_k\).
\end{lemma}

\begin{proof}
Since \(K_{ij}^*\) is blockwise constant and \(\gamma_i\) is constant on each sign group,
we have
\[
(L_k z)_i
=
-\gamma_+ z_i+\lambda_k\Bigl(a_+\sum_{j\in S_+}z_j+b_+\sum_{j\in S_-}z_j\Bigr),
\quad i\in S_+,
\]
and
\[
(L_k z)_i
=
-\gamma_- z_i+\lambda_k\Bigl(b_-\sum_{j\in S_+}z_j+a_-\sum_{j\in S_-}z_j\Bigr),
\quad i\in S_-.
\]

Let \(z\in V_+\). Then \(z_i=0\) on \(S_-\) and \(\sum_{i\in S_+}z_i=0\), so
\[
\sum_{j\in S_+}z_j=0,
\quad
\sum_{j\in S_-}z_j=0.
\]
Hence
\[
(L_k z)_i=-\gamma_+ z_i \quad (i\in S_+),
\quad
(L_k z)_i=0 \quad (i\in S_-).
\]
Moreover,
\[
\sum_{i\in S_+}(L_k z)_i
=
-\gamma_+\sum_{i\in S_+}z_i
=
0.
\]
Thus \(L_k z\in V_+\), and \(V_+\) is invariant.

The proof for \(V_-\) is identical, so we omit here.

Now let \(z\in V_{\mathrm{mean}}\). Then there exist \(u_+,u_-\in\mathbb R\) such that
\[
z_i=u_+ \quad (i\in S_+),
\quad
z_i=u_- \quad (i\in S_-).
\]
Therefore
\[
\sum_{j\in S_+}z_j=n_+u_+,
\quad
\sum_{j\in S_-}z_j=n_-u_-.
\]
Substituting into the formulas above, we see that \((L_k z)_i\) is constant on \(S_+\)
and constant on \(S_-\). Hence \(L_k z\in V_{\mathrm{mean}}\), so
\(V_{\mathrm{mean}}\) is invariant.

The direct-sum decomposition is immediate: for any \(z\in\mathbb R^n\), subtracting the
group means on \(S_+\) and \(S_-\) yields a unique decomposition into a zero-mean part on
\(S_+\), a zero-mean part on \(S_-\), and a blockwise constant part. In other words,
\(\mathbb R^n=V_+\oplus V_-\oplus V_{\mathrm{mean}}\).
\end{proof}

\begin{proposition}[Spectrum of the sign-split linearized operator]\label{prop:app_spectrum_sign_split}
Fix \(k\ne p\). The spectrum of \(L_k\) is given as follows.

On \(V_+\), the only eigenvalue is
\[
-\gamma_+
\]
with multiplicity \(n_+-1\).

On \(V_-\), the only eigenvalue is
\[
-\gamma_-
\]
with multiplicity \(n_--1\).

On \(V_{\mathrm{mean}}\), the eigenvalues are precisely those of the matrix
\[
B_k=
\begin{pmatrix}
\lambda_k n_+ a_+ - \gamma_+ & \lambda_k n_- b_+\\
\lambda_k n_+ b_- & \lambda_k n_- a_- - \gamma_-
\end{pmatrix}.
\]
\end{proposition}

\begin{proof}
By \Cref{prop:app_invariant_decomposition}, the decomposition
\[
\mathbb R^n=V_+\oplus V_-\oplus V_{\mathrm{mean}}
\]
is \(L_k\)-invariant.

If \(z\in V_+\), then, as shown in the proof of
\Cref{prop:app_invariant_decomposition},
\[
(L_k z)_i=-\gamma_+ z_i \quad (i\in S_+),
\quad
(L_k z)_i=0 \quad (i\in S_-).
\]
Hence
\[
L_k|_{V_+}=-\gamma_+ I,
\]
and therefore the only eigenvalue on \(V_+\) is \(-\gamma_+\), with multiplicity
\[
\dim V_+=n_+-1.
\]

Similarly,
\[
L_k|_{V_-}=-\gamma_- I,
\]
so the only eigenvalue on \(V_-\) is \(-\gamma_-\), with multiplicity
\[
\dim V_-=n_--1.
\]

It remains to analyze \(V_{\mathrm{mean}}\). Let \(z\in V_{\mathrm{mean}}\), so that
\[
z_i=u_+ \quad (i\in S_+),
\quad
z_i=u_- \quad (i\in S_-).
\]
Then
\[
\sum_{j\in S_+}z_j=n_+u_+,
\quad
\sum_{j\in S_-}z_j=n_-u_-.
\]
Substituting into the definition of \(L_k\), we obtain
\[
\dot u_+
=
-\gamma_+u_+ + \lambda_k(n_+a_+u_+ + n_-b_+u_-),
\]
\[
\dot u_-
=
-\gamma_-u_- + \lambda_k(n_+b_-u_+ + n_-a_-u_-).
\]
Equivalently,
\[
\binom{\dot u_+}{\dot u_-}
=
B_k\binom{u_+}{u_-}.
\]
Hence the eigenvalues of \(L_k\) on \(V_{\mathrm{mean}}\) are precisely the eigenvalues of
\(B_k\). Combining the three invariant pieces yields the result.
\end{proof}

%============================================================================================================================

\subsection{Proof of \Cref{stability_split}}\label{app_subsec:thm52}

\begin{proof}
By \Cref{lem:app_pure_mode_linearization}, the linearization around the sign-split equilibrium
\eqref{eq:sign_split_equilibrium} decouples mode by mode over the transverse eigendirections
\(e_k\), \(k\ne p\). For each such \(k\), the corresponding modewise system is
\eqref{eq:app_mode_system_sign_split} on \(\mathbb R^n\).

By \Cref{prop:app_invariant_decomposition}, this system admits the invariant decomposition
\[
\mathbb R^n=V_+\oplus V_-\oplus V_{\mathrm{mean}},
\]
and \Cref{prop:app_spectrum_sign_split} gives the corresponding spectrum. More precisely,
for each \(k\ne p\), the eigenvalues consist of
\[
-\gamma_+ \quad \text{with multiplicity } n_+-1,
\quad
-\gamma_- \quad \text{with multiplicity } n_--1,
\]
together with the two eigenvalues of \(B_k\).

Therefore the full linearization is linearly asymptotically stable if and only if all these
eigenvalues have negative real part. The scalar eigenvalues \(-\gamma_+\) and \(-\gamma_-\)
are negative if and only if
\[
\gamma_+>0,
\quad
\gamma_->0.
\]
For the \(2\times2\) block \(B_k\), both eigenvalues
have negative real part if and only if
\[
\operatorname{tr}(B_k)<0,
\quad
\det(B_k)>0.
\]
Hence the sign-split equilibrium \eqref{eq:sign_split_equilibrium} is linearly asymptotically
stable if and only if
\[
\gamma_+>0,\quad \gamma_->0,
\]
and, for every \(k\ne p\),
\[
\operatorname{tr}(B_k)<0,
\quad
\det(B_k)>0.
\]
This proves \Cref{stability_split}.
\end{proof}
    \section{Proofs for \Cref{sec_global}}\label{app_proof_6}
    In this appendix, we provide the technical proofs omitted from \Cref{sec_global}.
%============================================================================================================================
    \subsection{Detailed proof of \Cref{thm:cone_selection_e1}}\label{proof_thm61}
    
    \begin{proof}[Proof of \Cref{thm:cone_selection_e1}]
    We divide the proof into four steps.
    
    \medskip
    \noindent
    \textbf{Step 1 (Uniform positivity of the first mode):}
    Recall that
    \[
    \phi_i
    =
    \sum_{l=1}^d c_{i,l}\lambda_l\sum_{j=1}^nK_{ij}c_{j,l}
    =
    \left\langle
    x_i,\sum_{j=1}^nK_{ij}Vx_j
    \right\rangle.
    \]
    Since \(V\) is symmetric with eigenvalues $\{\lambda_i\}_{i=1}^d$, the assumption
    \[
    \lambda_1>\max_{k\ge2}|\lambda_k|
    \]
    implies that
    \[
    \|V\|_{\mathrm{op}}=\lambda_1,
    \]
    where \(\|\cdot\|_{\mathrm{op}}\) denotes the operator norm. Hence, for any unit vectors \(x,y\in\mathbb S^{d-1}\),
    \[
    |\langle x,Vy\rangle|
    \le \|V\|_{\mathrm{op}}
    =\lambda_1.
    \]
    Therefore,
    \begin{align}\label{philambda}
    \phi_i
    =
    \sum_{j=1}^nK_{ij}\langle x_i,Vx_j\rangle
    \le
    \sum_{j=1}^nK_{ij}\lambda_1
    =
    \lambda_1,
    \end{align}
    where we used \(K_{ij}\ge0\) and \(\sum_{j=1}^nK_{ij}=1\).
    
    Now define and set as follows:
    \[
    m(t):=\min_{1\le i\le n} c_{i,1}(t),\quad i_*(t)\in\argmin_{i\in[n]}c_{i,1}(t)
    \]
    Since \(c_{j,1}(t)\ge m(t)\) for all \(j\),
    \[
    \sum_{j=1}^nK_{i_*j}c_{j,1}(t)\ge m(t).
    \]
    Using the equation for \(c_{i,1}\), we obtain
    \[
    \dot c_{i_*,1}
    =
    \lambda_1\sum_{j=1}^nK_{i_*j}c_{j,1}-\phi_{i_*}c_{i_*,1}
    \ge
    \lambda_1 m(t)-\phi_{i_*}m(t)
    =
    (\lambda_1-\phi_{i_*})m(t)\ge0,
    \]
    where we use \eqref{philambda} in the last inequality.
    Thus the minimum \(m(t)\) is nondecreasing, and therefore
    \[
    c_{i,1}(t)\ge m(t)\ge m(0)\ge\delta,\quad\forall~i\in[n],~~t\ge0.    \]
    This proves that \(\mathcal C_\delta\) is forward invariant.
    
    \medskip
    \noindent
    \textbf{Step 2 (Evolution equation for the transverse ratios):}
    Fix \(k\neq1\), and define
    \[
    r_{i,k}:=\frac{c_{i,k}}{c_{i,1}}.
    \]
    By \textbf{Step 1}, the denominator stays uniformly positive, so \(r_{i,k}\) is well defined for all \(t\ge0\).
    
    We differentiate $r_{i,k}$ with substituting \eqref{C_ODE} to derive
    \[
    \dot r_{i,k}
    =
    \frac{1}{c_{i,1}}
    \sum_{j=1}^nK_{ij}c_{j,1}
    \bigl(\lambda_k r_{j,k}-\lambda_1 r_{i,k}\bigr).
    \]
    where the $\phi_i$-terms canceled out.
    Equivalently,
    \begin{align}\label{eq:app_ratio_dynamics_main}
    	\dot r_{i,k}
    	=
    	\alpha_i
    	\left(
    	\lambda_k\sum_{j=1}^n\omega_{ij}r_{j,k}
    	-\lambda_1 r_{i,k}
    	\right),
    \end{align}
    where
    \[
    \alpha_i:=
    \frac{\sum_{j=1}^nK_{ij}c_{j,1}}{c_{i,1}},
    \quad
    \omega_{ij}:=
    \frac{K_{ij}c_{j,1}}{\sum_{\ell=1}^nK_{i\ell}c_{\ell,1}}.
    \]
    Clearly, one can see that
    \[
    \omega_{ij}\ge0,
    \quad
    \sum_{j=1}^n\omega_{ij}=1.
    \]
    Moreover, using \(c_{j,1}\ge\delta\) and \(c_{i,1}\le1\), we get
    \[
    \alpha_i
    =
    \frac{\sum_{j=1}^nK_{ij}c_{j,1}}{c_{i,1}}
    \ge
    \sum_{j=1}^nK_{ij}c_{j,1}
    \ge
    \delta.
    \]
    
    \medskip
    \noindent
    \textbf{Step 3 (Exponential decay of the transverse ratios):}
    Fix \(k\neq1\), and for simplicity write
    \[
    r_i:=r_{i,k},
    \quad
    R(t):=\max_{1\le i\le n}|r_i(t)|.
    \]
    Since each \(r_i\) is \(C^1\), each \(|r_i|\) is locally Lipschitz. Therefore \(R(t)\) is locally Lipschitz, and hence differentiable for a.e.\ \(t\ge0\).
    
    Let \(i^*(t)\in[n]\) be such that
    \[
    |r_{i^*}(t)|=R(t).
    \]

    We consider two cases.

\emph{Case 1:} $r_{i^*}(t)=R(t)\ge 0$. Since $r_j(t)\le R(t)$ for all $j$ and
\[
\omega_{i^*j}\ge 0,
\quad
\sum_{j=1}^n \omega_{i^*j}=1,
\]
we have
\[
-R(t)\le \sum_{j=1}^n \omega_{i^*j} r_j(t)\le R(t),
\]
and therefore
\[
\lambda_k\sum_{j=1}^n \omega_{i^*j} r_j(t)\le |\lambda_k|\,R(t).
\]
Substituting $r_{i^*}(t)=R(t)$ into \eqref{eq:app_ratio_dynamics_main}, we obtain
\[
\dot r_{i^*}
=
\alpha_{i^*}
\left(
\lambda_k\sum_{j=1}^n \omega_{i^*j} r_j-\lambda_1 R(t)
\right)
\le
-\alpha_{i^*}(\lambda_1-|\lambda_k|)R(t).
\]
By Step~2, we have $\alpha_i\ge \delta$ for every $i\in[n]$, and in particular
$\alpha_{i^*}\ge \delta$. Hence
\[
\dot r_{i^*}(t)\le -\delta(\lambda_1-|\lambda_k|)R(t).
\]
Since in this case $|r_{i^*}|=r_{i^*}$, it follows that
\[
\frac{d}{dt}|r_{i^*}(t)|
=
\dot r_{i^*}(t)
\le
-\delta(\lambda_1-|\lambda_k|)R(t).
\]

\emph{Case 2:} $r_{i^*}(t)=-R(t)\le 0$. Then $r_j(t)\ge -R(t)$ for all $j$, and again
\[
-R(t)\le \sum_{j=1}^n \omega_{i^*j} r_j(t)\le R(t).
\]
Hence
\[
\lambda_k\sum_{j=1}^n \omega_{i^*j} r_j(t)\ge -|\lambda_k|\,R(t).
\]
Using $r_{i^*}(t)=-R(t)$ in \eqref{eq:app_ratio_dynamics_main}, we get
\[
\dot r_{i^*}
=
\alpha_{i^*}
\left(
\lambda_k\sum_{j=1}^n \omega_{i^*j} r_j+\lambda_1 R(t)
\right)
\ge
\alpha_{i^*}(\lambda_1-|\lambda_k|)R(t)
\ge
\delta(\lambda_1-|\lambda_k|)R(t).
\]
Since now $|r_{i^*}|=-r_{i^*}$, we obtain
\[
\frac{d}{dt}|r_{i^*}(t)|
=
-\dot r_{i^*}(t)
\le
-\delta(\lambda_1-|\lambda_k|)R(t).
\]

Collecting the two cases, we conclude that for a.e.\ $t\ge 0$,
\[
\dot R(t)\le -\delta(\lambda_1-|\lambda_k|)R(t).
\]
Hence, by Gr\"onwall's inequality,
\[
R(t)\le R(0)e^{-\delta(\lambda_1-|\lambda_k|)t}.
\]
    
    We consider two cases.
    
    \medskip
    \noindent
    \emph{Case 1 (\(r_{i^*}(t)=R(t)\ge0\)):}
    Since \(r_j(t)\le R(t)\) and \(\omega_{i^*j}\ge0\) for all $j$ with \(\sum_j\omega_{i^*j}=1\), we have
    \[
    -R(t)\le \sum_{j=1}^n\omega_{i^*j}r_j(t)\le R(t),
    \]
    which leads to
    \[
    \lambda_k\sum_{j=1}^n\omega_{i^*j}r_j(t)\le |\lambda_k|R(t).
    \]
    Substituting \(r_{i^*}(t)=R(t)\) in \eqref{eq:app_ratio_dynamics_main}, we obtain
    \[\dot r_{i^*}=\alpha_{i^*}\left(\lambda_k\sum_{j=1}^n \omega_{i^*j} r_j-\lambda_1 R(t)\right)\le-\alpha_{i^*}(\lambda_1-|\lambda_k|)R(t).\]
    By \textbf{Step 2}, we have $\alpha_i\ge \delta$ for every $i\in[n]$, and in particular
    $\alpha_{i^*}\ge \delta$. Hence
    \begin{align}\label{rdot}
        \dot r_{i^*}(t)\le -\delta(\lambda_1-|\lambda_k|)R(t).
    \end{align}
    Since in this case \(|r_{i^*}|=r_{i^*}\), it follows that
    \[
    \frac{d}{dt}|r_{i^*}(t)|
    =
    \dot r_{i^*}(t)
    \le
    -\delta(\lambda_1-|\lambda_k|)R(t).
    \]
    
    \medskip
    \noindent
    \emph{Case 2 (\(r_{i^*}(t)=-R(t)\le0\)):}
    Then \(r_j(t)\ge -R(t)\) for all \(j\), and again
    \[
    -R(t)\le \sum_{j=1}^n\omega_{i^*j}r_j(t)\le R(t).
    \]
    Therefore,
    \[
    \lambda_k\sum_{j=1}^n\omega_{i^*j}r_j(t)\ge -|\lambda_k|R(t).
    \]
    Using \(r_{i^*}(t)=-R(t)\) in \eqref{eq:app_ratio_dynamics_main}, we get
    \[
    \dot r_{i^*}
    =
    \alpha_{i^*}
    \left(
    \lambda_k\sum_{j=1}^n\omega_{i^*j}r_j+\lambda_1R(t)
    \right)
    \ge
    \alpha_{i^*}(\lambda_1-|\lambda_k|)R(t)
    \ge
    \delta(\lambda_1-|\lambda_k|)R(t).
    \]
    Since now \(|r_{i^*}|=-r_{i^*}\), we have the same result in \eqref{rdot}.
    
    \medskip
    Collecting the results from two cases above, one obtains
    \[\dot R(t)\le -\delta(\lambda_1-|\lambda_k|)R(t)\quad\text{for a.e. }t\ge0.\]
    Hence, by Gronwall's inequality,
    \[
    R(t)\le R(0)e^{-\delta(\lambda_1-|\lambda_k|)t}.
    \]
    Thus, for every \(k\neq1\) with $R_k(0):=\max_{i\in[n]}|r_{i,k}(0)|$, we derive
    \[
    \max_{1\le i\le n}\left|\frac{c_{i,k}(t)}{c_{i,1}(t)}\right|
    \le
    R_k(0)e^{-\delta(\lambda_1-|\lambda_k|)t}.
    \]
    
    \medskip
    \noindent
    \textbf{Step 4 (Convergence to the first mode):}
    For each \(k\neq1\), \textbf{Step 3} shows that
    \[
    \left|\frac{c_{i,k}(t)}{c_{i,1}(t)}\right|\to0
    \quad\text{uniformly in }i,
    \]
    with
    \[
    \lim_{t\to\infty}|c_{i,k}(t)|
    \le
    \lim_{t\to\infty}R_k(0)e^{-\delta(\lambda_1-|\lambda_k|)t}=0,
    \quad \forall~k\ne1.
    \]
    By \(\|x_i(t)\|=1\), this leads to 
    \[
    c_{i,1}(t)^2\to1.
    \]
    On the other hand, \textbf{Step 1} gives \(c_{i,1}(t)\ge\delta>0\), so the negative branch is excluded and therefore
    \[
    c_{i,1}(t)\to1.
    \]
    Consequently,
    \[
    \lim_{t\to\infty}x_i(t)=\lim_{t\to\infty}\sum_{k=1}^d c_{i,k}(t)e_k=e_1,
    \quad\forall~i\in[n],
    \]
    which completes the proof.
    \end{proof}
    %============================================================================================================================

    \subsection{Detailed proof of \Cref{rho_monotone_two_particle}}\label{proof_lem63}
    \begin{proof}[Proof of \Cref{rho_monotone_two_particle}]
    Set
    \[
    x:=x_1,\quad y:=x_2,\quad \rho:=\langle x,y\rangle.
    \]
    Since \(V\) is symmetric and negative definite, the matrix
    \[
    B:=-V
    \]
    is symmetric and positive definite. We introduce
    \[
    s:=x+y,\quad d:=x-y.
    \]
    to rewrite
    \[
    x=\frac{s+d}{2},\quad y=\frac{s-d}{2},
    \]
    and
    \[
    \langle s,d\rangle=0,\quad |s|^2=2(1+\rho),\quad |d|^2=2(1-\rho).
    \]
    In particular, if \(-1<\rho<1\), then \(s\ne0\) and \(d\neq 0\).
    
    Define the three scalar quantities
    \[
    A(s):=\langle s,Bs\rangle,\quad
    C(s,d):=\langle s,Bd\rangle,\quad
    D(d):=\langle d,Bd\rangle.
    \]
    We omit the dependency for the notational simplicity. Since the matrix \(B\) is positive definite, we have
    \[
    A>0,\quad D>0.
    \]
    Moreover, by the Cauchy--Schwarz inequality for the inner product induced by \(B\),
    \[
    C^2\le AD.
    \]
    Hence there exists \(E(s,d)\ge0\) such that
    \[
    A=\frac{C^2}{D}+E.
    \]
    
    We next compute the three attention scores
    \[
    a(x):=\langle x,Vx\rangle,\quad
    b(x,y):=\langle x,Vy\rangle=\langle y,Vx\rangle,\quad
    c(y):=\langle y,Vy\rangle.
    \]
    Since \(V=-B\) and \(x=(s+d)/2\), \(y=(s-d)/2\), one obtains
    \begin{align*}
    a
    &=\left\langle \frac{s+d}{2},\,V\frac{s+d}{2}\right\rangle
     =-\frac14\langle s+d,B(s+d)\rangle
     =-\frac14(A+2C+D),\\
    b
    &=\left\langle \frac{s+d}{2},\,V\frac{s-d}{2}\right\rangle
     =-\frac14\langle s+d,B(s-d)\rangle
     =-\frac14(A-D),\\
    c
    &=\left\langle \frac{s-d}{2},\,V\frac{s-d}{2}\right\rangle
     =-\frac14\langle s-d,B(s-d)\rangle
     =-\frac14(A-2C+D).
    \end{align*}
    
    For the two-particle system, the softmax weights are
    \[
    w_{11}=\frac{e^{\beta a}}{e^{\beta a}+e^{\beta b}},
    \quad
    w_{12}=\frac{e^{\beta b}}{e^{\beta a}+e^{\beta b}},
    \quad
    w_{21}=\frac{e^{\beta b}}{e^{\beta b}+e^{\beta c}},
    \quad
    w_{22}=\frac{e^{\beta c}}{e^{\beta b}+e^{\beta c}}.
    \]
    Denoting
    \[
    \eta_x:=w_{11}-w_{12},
    \quad
    \eta_y:=w_{21}-w_{22}.
    \]
    one can rewrite it as
    \[
    w_{11}=\frac{1+\eta_x}{2},\quad
    w_{12}=\frac{1-\eta_x}{2},\quad
    w_{21}=\frac{1+\eta_y}{2},\quad
    w_{22}=\frac{1-\eta_y}{2}.
    \]
    
    Since
    \[
    \dot x=P_x^\perp(w_{11}Vx+w_{12}Vy),\quad
    \dot y=P_y^\perp(w_{21}Vx+w_{22}Vy),
    \]
    we get
    \begin{align*}
    \dot\rho=w_{11}(b-\rho a)+w_{12}(c-\rho b)+w_{21}(a-\rho b)+w_{22}(b-\rho c).
    \end{align*}
    A direct computation from the above expressions for \(a,b,c\) yields
    \begin{align*}
    b-\rho a
    &=\frac{-(1-\rho)A+2\rho C+(1+\rho)D}{4},\\
    c-\rho b
    &=\frac{-(1-\rho)A+2C-(1+\rho)D}{4},\\
    a-\rho b
    &=\frac{-(1-\rho)A-2C-(1+\rho)D}{4},\\
    b-\rho c
    &=\frac{-(1-\rho)A-2\rho C+(1+\rho)D}{4}.
    \end{align*}
    Substituting these identities and the representation of the weights in terms of \(\eta_x,\eta_y\), we obtain
    \begin{align}\label{eq:app_rho_ADC_eta}
    \dot\rho
    =
    -\frac{1-\rho}{2}A
    -\frac{1+\rho}{4}D(\eta_y-\eta_x)
    -\frac{1-\rho}{4}C(\eta_x+\eta_y).
    \end{align}
    
    We now compute \(\eta_x\) and \(\eta_y\) explicitly. Since
    \[
    a-b=-\frac{C+D}{2},
    \quad
    b-c=\frac{D-C}{2},
    \]
    we get
    \[
    \eta_x
    =
    \frac{e^{\beta a}-e^{\beta b}}{e^{\beta a}+e^{\beta b}}
    =
    \tanh\!\left(\frac{\beta}{2}(a-b)\right)
    =
    -\tanh\!\left(\frac{\beta}{4}(C+D)\right),
    \]
    and
    \[
    \eta_y
    =
    \frac{e^{\beta b}-e^{\beta c}}{e^{\beta b}+e^{\beta c}}
    =
    \tanh\!\left(\frac{\beta}{2}(b-c)\right)
    =
    \tanh\!\left(\frac{\beta}{4}(D-C)\right).
    \]
    Setting
    \[
    p:=\frac{\beta}{4}(C+D),\quad q:=\frac{\beta}{4}(D-C),
    \quad
    \Delta:=\cosh p\,\cosh q,
    \]
    we obtain
    \begin{align}\label{eta_relation}
    \begin{aligned}
    \eta_y-\eta_x
    &=\tanh q+\tanh p
     =\frac{\sinh(\beta D/2)}{\Delta},\\
    \eta_x+\eta_y
    &=-\tanh p+\tanh q
     =-\frac{\sinh(\beta C/2)}{\Delta},
    \end{aligned}
    \end{align}
    where we used
    \[
    \tanh q+\tanh p=\frac{\sinh(p+q)}{\cosh p\,\cosh q},
    \quad
    \tanh q-\tanh p=\frac{\sinh(q-p)}{\cosh p\,\cosh q},
    \]
    Substituting \eqref{eta_relation} into \eqref{eq:app_rho_ADC_eta}, we have the following equation:
    \begin{align}\label{eq:app_rho_pre_r_t}
    \dot\rho
    =
    -\frac{1-\rho}{2}A
    -\frac{1+\rho}{4\Delta}D\sinh\!\left(\frac{\beta D}{2}\right)
    +\frac{1-\rho}{4\Delta}C\sinh\!\left(\frac{\beta C}{2}\right).
    \end{align}
    
    Next, set
    \[
    r:=\frac{\beta D}{2}>0,
    \quad
    t:=\frac{\beta C}{2}\in\mathbb R.
    \]
    Since
    \[
    \Delta
    =
    \cosh\!\left(\frac{r+t}{2}\right)\cosh\!\left(\frac{r-t}{2}\right)
    =
    \frac{\cosh r+\cosh t}{2},
    \]
    and
    \[
    A=\frac{C^2}{D}+E=\frac{2t^2}{\beta r}+E,
    \]
    the equation \eqref{eq:app_rho_pre_r_t} becomes
    \begin{align}\label{rhodot}
    \dot\rho
    =
    -\frac{
    (1-\rho)\beta rE(\cosh r+\cosh t)
    +2(1+\rho)r^2\sinh r
    +2(1-\rho)\Xi_r(t)
    }{
    2\beta r(\cosh r+\cosh t)
    },
    \end{align}
    where
    \[
    \Xi_r(t):=t^2(\cosh r+\cosh t)-rt\sinh t.
    \]
    
    It remains to show that
    \begin{align}\label{Xi_ineq}
    \Xi_r(t)\ge0
    \quad\text{for all }r>0,\ t\in\mathbb R.
    \end{align}
    Since \(\Xi_r\) is even in \(t\), it is enough to consider \(t\ge0\).
    
    If \(0\le t\le r\), as the function \(x\mapsto \sinh x/x\) is increasing on \((0,\infty)\), one has
    \[
    r\sinh t\le t\sinh r.
    \]
    Therefore, with $\cosh r-\sinh r=e^{-r}>0$,
    \[
    rt\sinh t\le t^2\sinh r<t^2(\cosh r+\cosh t),
    \]
    which implies \(\Xi_r(t)>0\).
    
    If \(t\ge r\), then
    \[
    r\sinh t\le t\sinh t,
    \]
    and thus
    \[
    rt\sinh t\le t^2\sinh t<t^2(\cosh r+\cosh t),
    \]
    again implying \(\Xi_r(t)>0\).
    
    Hence the relation \eqref{Xi_ineq} holds with equality only at \(t=0\).
    
    \medskip
    Returning to the expression \eqref{rhodot} for \(\dot\rho\), we note that each term in the numerator is nonnegative, and the second term is in fact strictly positive because
    \[
    -1<\rho<1,\quad r>0,\quad \sinh r>0.
    \]
    Therefore the numerator is strictly positive, and consequently
    \[
    \dot\rho<0.
    \]
    This is the desired result.
    \end{proof}
    %============================================================================================================================
    \subsection{Detailed proof of \Cref{ae_selection_most_negative}}\label{proof_thm65}

    \begin{proof}[Proof of \Cref{ae_selection_most_negative}]
    Let \(\Phi_t\) denote the flow generated by \eqref{TF_ODE_sym} on
    \[
    \mathcal X:=\mathbb S^{d-1}\times \mathbb S^{d-1}.
    \]
    We also introduce the diagonal set
    \[
    \Delta:=\{(x,x):x\in\mathbb S^{d-1}\}\subset\mathcal X.
    \]
    The set \(\Delta\) is exactly the manifold of homogeneous configurations. Since \(\Delta\) is a smooth submanifold of positive codimension in \(\mathcal X\), it has measure zero. Thus it suffices to consider initial data
    \[
    z_0=(x_1(0),x_2(0))\in\mathcal X\setminus\Delta.
    \]
    For such an initial condition, the corresponding solution is nontrivial. Therefore, by \Cref{rho_monotone_two_particle}, the pairwise correlation
    \[
    \rho(t):=\langle x_1(t),x_2(t)\rangle
    \]
    is strictly decreasing on \((-1,1)\), and hence
    \[
    \rho(t)\to -1
    \quad\text{as }t\to\infty
    \]
    by a standard \(\omega\)-limit set argument.
    Therefore
    \[
    |x_1(t)+x_2(t)|^2
    =
    |x_1(t)|^2+|x_2(t)|^2+2\langle x_1(t),x_2(t)\rangle
    =
    2(1+\rho(t))
    \to 0,
    \]
    so every \(\omega\)-limit set is contained in the sign-split manifold $\mathcal M_{\mathrm{bbp}}$.
    
    Let \(\Omega:=\omega(z_0)\) denote the \(\omega\)-limit set of \(z_0\). Since \(\mathcal X\) is compact, \(\Omega\) is nonempty, compact, connected, and invariant.
    By the previous argument,
    \[
    \Omega\subset \mathcal M_{\mathrm{bbp}}.
    \]
    
    \medskip
    \noindent\textbf{Step 1 (Reduction of the \(\omega\)-limit set to a single eigenspace):}
    On \(\mathcal M_{\mathrm{bbp}}\), write
    \[
    x_1=u,\quad x_2=-u,\quad
    u=\sum_{k=1}^d u_k e_k,\quad p_k:=u_k^2.
    \]
    Then the induced flow on \(\mathcal M_{\mathrm{bbp}}\) is exactly the reduced polarized
    dynamics from \Cref{subsec:bipolar_mfd}:
    \[
    \dot p_k = 2p_k\,\alpha(M)(\lambda_k-M),
    \quad
    M=\sum_{k=1}^d \lambda_k p_k,
    \quad
    \alpha(M)=\tanh(\beta M).
    \]
    Since \(V\) is negative definite,
    \[
    M(u)=\langle u,Vu\rangle<0
    \quad\text{for every }u\in\mathbb S^{d-1},
    \]
    and therefore, by \Cref{lem:M_monotone},
    \[
    \dot M
    =
    2\alpha(M)\sum_{k=1}^d p_k(\lambda_k-M)^2
    \le 0.
    \]
    Moreover, \(\dot M=0\) holds if and only if all active modes correspond to the same
    eigenvalue, that is, if and only if
    \[
    u\in E_\lambda\cap\mathbb S^{d-1}
    \quad\text{for some }\lambda\in\sigma(V),
    \]
    where \(E_\lambda:=\ker(V-\lambda I)\).
    
    Hence \(M\) is a Lyapunov function for the reduced flow on
    \(\mathcal M_{\mathrm{bbp}}\), and LaSalle's invariance principle yields
    \[
    \Omega\subset \mathcal E:=\bigcup_{\lambda\in\sigma(V)}\mathcal E_\lambda,
    \quad
    \mathcal E_\lambda:=\{(u,-u):u\in E_\lambda\cap\mathbb S^{d-1}\}.
    \]
    Since \(\Omega\) is connected, there exists \(\lambda_*\in\sigma(V)\) such that
    \[
    \Omega \subset \mathcal E_{\lambda_*}.
    \]
    
    If \(\lambda_*=\lambda_d\), then the simplicity of \(\lambda_d\) implies
    \[
    \mathcal E_{\lambda_d}=\{(e_d,-e_d),\,(-e_d,e_d)\}.
    \]
    Because \(\Omega\) is connected, it follows that \(\Omega\) is a singleton. Hence
    \[
    \Phi_t(z_0)\to (e_d,-e_d)
    \quad\text{or}\quad
    \Phi_t(z_0)\to (-e_d,e_d).
    \]
    Therefore, it remains to show that the set of initial conditions for which
    \[\Omega\subset \mathcal E_\lambda
    \quad\text{for some }\lambda\in\sigma(V)~\mbox{with}~\lambda>\lambda_d,\]
    has measure zero.
    
    \medskip
    \noindent\textbf{Step 2 (Measure-zero exceptional set):} Fix an eigenvalue \(\lambda>\lambda_d\). Then
    \[
    \mathcal E_\lambda=\{(u,-u):u\in E_\lambda\cap\mathbb S^{d-1}\}
    \]
    is a compact equilibrium manifold. Let \(z_*=(u_*,-u_*)\in\mathcal E_\lambda\). By orthogonal invariance of \eqref{TF_ODE_sym}, the linearization at \(z_*\) is conjugate to the linearization at a sign-split equilibrium \((\widetilde e_1,-\widetilde e_1)\) for a diagonal interaction matrix with \(\widetilde e_1=u_*\) and \(\widetilde e_d=e_d\).

    Now apply the sign-split linearization from \Cref{stability_split} in the case \(n=2\). Since \(n_+=n_-=1\), there are no within-group fluctuation modes, and the transverse mode \(\widetilde e_d\) yields the eigenvalue
    \[
    \mu_{d,2}=(\lambda_d-\lambda)\tanh(\beta\lambda)>0,
    \]
    because \(\lambda_d<\lambda<0\) and \(\tanh(\beta\lambda)<0\). Hence every point of \(\mathcal E_\lambda\) has a nontrivial unstable direction, while the zero eigenvalues are precisely tangent to \(\mathcal E_\lambda\). Therefore \(\mathcal E_\lambda\) is a compact normally hyperbolic invariant manifold with nontrivial unstable bundle. By the stable manifold theorem for normally hyperbolic invariant manifolds, its basin of attraction is contained in a smooth manifold of positive codimension in \(\mathcal X\), and hence has measure zero.

    If \(\omega(z_0)\subset\mathcal E_\lambda\), then necessarily \(\operatorname{dist}(\Phi_t(z_0),\mathcal E_\lambda)\to0\), so \(z_0\) belongs to that measure-zero set. Since \(V\) has only finitely many eigenvalues, the union over all \(\lambda>\lambda_d\) is still measure zero. Together with \textbf{Step 1}, this proves that, for almost every initial condition,
    \[
    \Phi_t(z_0)\to (e_d,-e_d)
    \quad\text{or}\quad
    \Phi_t(z_0)\to (-e_d,e_d).
    \]
    This proves the theorem.
    \end{proof}
\end{document}